\documentclass[numbers, webpdf]{ima-authoring-template}

\usepackage{algpseudocode}
\usepackage{algorithm}
\usepackage{algorithmicx}
\usepackage{float}

\allowdisplaybreaks
\graphicspath{{Fig/}}

\theoremstyle{thmstyletwo}%
\newtheorem{theorem}{Theorem}
\newtheorem{proposition}[theorem]{Proposition}%
\newtheorem{lemma}[theorem]{Lemma}%

\newtheorem{remark}{Remark}%
\newtheorem{assumption}{Assumption}%

\newcommand{\R}{\mathbb{R}}
\newcommand{\N}{\mathbb{N}}

\newcommand{\E}{\mathbb{E}}

\numberwithin{equation}{section}

\begin{document}

\copyrightyear{31th March 2025}
\appnotes{Paper}


\title[Efficient solution of ill-posed integral equations through averaging]{Efficient solution of ill-posed integral equations through averaging}

\author{Michael Griebel
\address{\orgdiv{Institut f\"ur Numerische Simulation}, \orgname{Universit\"at Bonn}, \orgaddress{\street{Friedrich-Hirzebruch-Allee 7}, \postcode{53115 Bonn}, \country{Germany}}}
\address{ \orgname{Fraunhofer Institute for Algorithms and Scientific Computing SCAI}, \orgaddress{\street{Schloss Birlinghoven}, \postcode{53754 Sankt Augustin}, \country{Germany}}}}
\author{Tim Jahn*
\address{\orgdiv{Institut f\"ur Mathematik}, \orgname{TU Berlin}, \orgaddress{\street{	Stra\ss{}e des 17. Juni 135}, \postcode{ 10623 Berlin }, \country{Germany}}}}

\authormark{Griebel and Jahn}

\corresp[*]{Corresponding author: \href{email:jahn@tu-berlin.de}{jahn@tu-berlin.de}}

\received{31th}{3}{2025}


\abstract{	This paper discusses the error and cost aspects of ill-posed integral equations when given discrete noisy point evaluations on a fine grid. Standard solution methods usually employ discretization schemes that are directly induced by the measurement points. Thus, they may scale unfavorably with the number of evaluation points, which can result in computational inefficiency. To address this issue, we propose an algorithm that achieves the same level of accuracy while significantly reducing computational costs. Our approach involves an initial averaging procedure to sparsify the underlying grid. To keep the exposition simple, we focus on regularization via the truncated singular value decomposition of one-dimensional ill-posed integral equations that have sufficient smoothness. However, the approach can be generalized to other popular regularization methods and more complicated two- and three-dimensional problems with appropriate modifications.}


\maketitle

\section{Introduction}\label{s1}

This article discusses integral equations. Examples for their application include various types of tomography, such as x-ray, electrical impedance, or magnetic resonance tomography in medical contexts. In geophysical applications and the oil industry, seismic waves are used to detect underground structures, resulting in an inverse problem. Additional examples are the determination of the implicit volatility that drives stock prices in financial mathematics or the 
 identification of the atmospheric state required for
weather predictions. In many cases, the desired object of determination is the distribution of a physical parameter, such as density, conductivity, or volatility, for spatially extended objects that are described by numerous variables. Therefore, the discretization of the integral equation often results in a high number of measured data and degrees of freedom, making the cost of approximative solution algorithms a significant factor. 
 The integral equation must be solved for a function that is not directly accessible but is measured in some way. Measurement models distinguishes between linear information, where Fourier modes or similar functionals of the data are accessible, and point-wise (standard) information, where only point evaluations of the data are available. This article focuses on point-wise information. Finding an appropriate discretization scheme can be challenging and often requires additional a priori information about the problem and the unknown solution. If the required information is unavailable, a common approach is to use the measurement grid as the discretization grid. The main idea of this article is to decrease the size of such an initially given fine discretization by averaging. Averaging data is a common engineering practice with many applications, see \cite{lyons2004understanding}. It has been successfully applied in the closely related field of numerical differentiation by finite differences, as shown in \cite{anderssen1999numerical}. Numerical differentiation is often a good starting point for understanding new approaches to inverse problems, as it captures essential features and challenges while being analytically simple, see \cite{hanke2001inverse}. We first introduce our new numerical method for such a model differentiation problem. Unlike \cite{anderssen1999numerical}, we address the problem directly within the framework of regularization theory, which enables the use of various regularization schemes and generalizes to integral equations. Additionally, we focus in this article solely on the one-dimensional case. This allows us to present the fundamental concepts and specific properties of our method in a clear and concise manner, without the complications that arise from a higher-dimensional setting. However, as we will discuss in the final section of this article, our approach can be applied to a wider range of problems.

This paper is structured as follows: Section \ref{s2} describes the general problem, while Section \ref{s3} introduces our novel approach and presents the main results. Specifically, we demonstrate that the optimal error rate can be achieved with significantly reduced cost complexity. We first illustrate this property using a simple example, for which we explicitly know the singular value decomposition of the discretized integral operator. Then, the method is extended to more general settings where additionally quadrature methods are employed to compute the approximations. The subject of adaptive regularization is also discussed. Section \ref{s6} presents numerical experiments, while Section \ref{s7} provides an outlook on how to generalize the results in various ways. The proofs of our findings are presented in detail in the appendix.
\section{The setting}\label{s2}
We will first formulate and analyze our approach for the exemplary integral equation
\begin{equation}\label{s1:e0}
	Kf=g,
\end{equation}
where    
\begin{equation}\label{s1:e1}
	g(x)=(Kf)(x)=\int_0^1\kappa(x,y)f(y)\mathrm{d}x,
\end{equation}
with the integral kernel 
\begin{equation}\label{s1:e1a}
	\kappa(x,y):=\min\left(x(1-y),y(1-x)\right).
\end{equation}
The rule \eqref{s1:e1} defines a compact operator $K:L^2(0,1)\to L^2(0,1)$. Obviously, this operator $K$ is not surjective, since $g=Kf$ is differentiable for any $f\in L^2(0,1)$. We furthermore assume that the exact $g$ is corrupted by (irregular) measurement noise. More precisely, we deal with the associated inverse problem of reconstructing $f\in L^2(0,1)$ from noisy point evaluations of $g$. These $m\in\N$  evaluations are taken  on a set of points $\xi_{lm}\in(0,1), l=1,...,m,$ and are given by
\begin{equation}\label{s1:e2}
	\begin{pmatrix}g^\delta_m\end{pmatrix}_{l=1}^m:=\begin{pmatrix}g(\xi_{lm})+\delta Z_l\end{pmatrix}_{l=1}^m \in\R^m,
\end{equation}
where $g=Kf$ is the exact (unknown) data, $\delta>0$ is the noise level and $Z_l$ are i.i.d random variables that are unbiased ($\E[Z_l]=0$) with unit variance $\E[Z_l^2]=1$.  Moreover, the continuity of $\kappa$ implies that $Kf$ is continuous for all $f\in L^2(0,1)$, which justifies the use of point evaluations. For simplicity, we assume that the evaluation points are uniformly distributed for this example, i.e., $\xi_{lm}:=l/(m+1)$, and thus form a uniform grid. The central task is now to solve the equation \eqref{s1:e0} from the noisy point evaluations of $g$. This can be formulated as follows:
\begin{equation*}
	\mbox{Given noisy measurements } g^\delta_m\in\R^m, \mbox{ find (approximate) } f.
\end{equation*}
This is a classic example of an inverse problem. A characteristic challenge of inverse problems is that they are often ill-posed in the Hadamard sense: A problem is said to be well-posed when it admits a unique solution that is continuously dependent on the input data, and it is said to be ill-posed when it is not well-posed. The above infinite dimensional equation \eqref{s1:e0} is ill-posed. While it can be shown that $K$ is injective, the domain of its inverse $K^{-1}$ is strictly less than $L^2(0,1)$ (since $K$ is not surjective), which implies that $K^{-1}$ is not continuous on its domain. Since we have to solve \eqref{s1:e0} from noisy data, this non-continuity of $K^{-1}$ implies that standard inversion methods are not feasible here, and we have to use additional regularization. Furthermore, we need to discretize. To construct an approximation for $f$, we proceed as follows: The measurement grid naturally induces a semi-discrete model of $K$, which we will denote by $K_m$. It is defined as
\begin{align}\label{s1:e2:a}
	K_m:L^2(0,1)&\to\R^m\\
	f &\mapsto \left( (Kf)(\xi_{lm})\right)_{l=1}^m = \left(\int \kappa(\xi_{lm},y) f(y)\mathrm{d}y\right)_{l=1}^m.
\end{align}
This gives us the semi-discrete equation (with noisy measurements)
\begin{equation}\label{s1:e2a}
	g_m^\delta= K_mf
\end{equation}
 which we have to solve. Obviously, even though $K$ is injective, the semi-discrete operator $K_m$ is not, i.e. solving \eqref{s1:e2a} in all of $L^2(0,1)$ is not sufficient to uniquely specify the solution. A natural further choice is to pick from the set of solutions of \eqref{s1:e2a} the element of minimal norm. This is equivalent to restricting the solution space to $\mathcal{N}(K_m)^\perp = \mathcal{R}(K_m^*)\subset L^2(0,1)$.   Note that for this specific setting, $K_m$ has full rank and thus ${\rm dim}\mathcal{N}(K_m)^\perp=m$.  Now the ill-posedness of the infinite-dimensional operator $K$ is inherited by the semi-discrete operator $K_m$ in the sense that it is highly ill-conditioned. Thus, one has to resort to regularization methods, such as Tikhonov regularization, Landweber iteration, or spectral cut-off techniques.   There is a large body of literature devoted to the treatment of the semi-discrete equation \eqref{s1:e2a}, see \cite{bertero1985linear, bertero1988linear, nair2007regularized, krebs2009sobolev, pereverzev2012balancing}, where various a priori and adaptive a posteriori regularization methods are discussed.
 
 In this paper, as a new contribution, we establish data compression as a means to reduce complexity, where we exploit the random cancellation of the averaged initial data. We also propose a new analysis of the discretization error, which is shown to be an improvement in certain cases, see Remark \ref{rem:disc:err} below. Finally, we give rigorous error bounds for an approach based on a design matrix containing only point evaluations instead of integrals of the kernel $\kappa$. In this article we concentrate on spectral cut-off, but our results could be generalized to other methods as well,   see Section \ref{s7} .

Spectral cut-off is based on the singular value decomposition of $K_m$, which we denote by $(v_{jm},u_{jm},\sigma_{jm})_{j\le m}$. The singular values $(\sigma_{jm})_{j\le m}$ form a positive decreasing sequence $\sigma_{1m}\ge ... \ge \sigma_{mm} >0$ and the singular functions $(v_{jm})_{j\le m}$ and vectors $(u_{jm})_{j\le m}$ are orthonormal bases in  $\mathcal{N}(K_m)^\perp\subset L^2(0,1)$ and $\R^m$, respectively. In addition, the following two relations hold:
\begin{align*}
	K_m v_{jm} =\sigma_{jm} u_{jm}\qquad \mbox{and}\qquad K_m^* u_{jm}=\sigma_{jm} v_{jm},
\end{align*} 
for all $j=1,...,m$, where $K_m^*:\R^m\to L^2(0,1)$ is the adjoint operator implicitly given by the relation $(f,K_m^*\lambda)_{L^2(0,1)} = (K_mf,\lambda)_{\R^m}$, for $f\in L^2(0,1)$ and $\lambda\in\R^m$.  We emphasize here that, while every compact operator between Hilbert spaces has a singular value decomposition, the concrete singular values and vectors (or functions) for a particular setting are rarely known   and usually have to be somehow numerically approximated. The specific setting we consider in the first part of this paper is exceptional in the sense that we can derive the singular value decomposition explicitly. In the second part of this article we study more general integral equations in Section \ref{s5} where we have to rely on numerical approximation methods.

Based on the singular value decomposition we now define an approximation to the unknown $f$ via spectral cut-off by
\begin{equation}\label{s1:e4}
	f^\delta_{k,m}:= \sum_{j=1}^k \frac{\left(g^\delta_m,u_{jm}\right)_{\R^m}}{\sigma_{jm}} v_{jm},
\end{equation}
where $k=0,...,m$ is the truncation index which has to be chosen depending on the measurement $g^\delta_m$ and the noise level $\delta>0$. To determine the singular value decomposition needed for $f^\delta_{k,m}$ we rely on a result from \cite{jahn2022cross}. There it was shown that the singular value decomposition of the semi-discrete operator $K_m:L^2(0,1)\to \R^m$ is closely related to the eigenvalue decomposition of a representation matrix. In fact, recall that we restricted the search space for the solution of \eqref{s1:e2a} to $\mathcal{N}(K_m)^\perp$. This space is finite-dimensional and has the form
\begin{equation}\label{ansatz}
	\mathcal{N}(K_m)^\perp:=\left\{ \sum_{l=1}^m \alpha_l \kappa(\xi_{lm},\cdot)~:~ \alpha_l\in\R, l=1,...,m\right\}.
\end{equation}
The basis functions $\kappa(\xi_{lm},\cdot)$ are exactly the Riesz representers of the bounded linear functionals $f\mapsto(Kf)(\xi_{lm})$, since $(Kf)(\xi_{lm}) = (f,\kappa(\xi_{lm},\cdot))$.
Now, expressing the action of $K_m$ in the basis $\{\kappa(\xi_{lm},\cdot), l=1,...,m\}$ yields 
\begin{equation*}
 K_m\left(\sum_{l=1}^m \alpha_l \kappa(\xi_{lm},\cdot)\right) = \sum_{l=1}^m \alpha_l \begin{pmatrix}\int \kappa(\xi_{1m},y) \kappa(\xi_{lm},y){\rm d} y \\ ... \\ \int \kappa(\xi_{mm},y) \kappa(\xi_{lm},y){\rm d} y\end{pmatrix} = T_m \alpha, 
\end{equation*}
with the matrix $T_m\in\R^{m\times m}$, where
\begin{equation}\label{s1:e2b}
	\left(T_m\right)_{ij}:=\left(\int \kappa(\xi_{im},y)\kappa(\xi_{jm},y){\rm d}y\right)_{ij},
\end{equation}
    and the vector $\alpha=\begin{pmatrix}\alpha_1 &  ... &\alpha_m\end{pmatrix}^T$. Note that due to the special form of the kernel $\kappa$ we can explicitly evaluate the entries of the matrix $T_m$. Furthermore, $T_m$ is symmetric by construction. The following relationship between the eigenvalue decomposition of $T_m$ and the singular value decomposition of $K_m$ is shown in \cite{jahn2022cross}:
\begin{proposition}\label{prop0}
	Let $(\lambda_{jm},w_{jm})_{j=1}^m$ with $\lambda_{1m}\ge ... \ge \lambda_{mm}>0$ denote the eigenvalue decomposition of $T_m$, i.e. $T_m=W_m\Lambda_m W_m^T$, where $w_{jm}$ indicates the $j$-th column of $W_m$ and $\Lambda_m$ is diagonal with the $j$-th diagonal entry $\lambda_{jm}$. Then, for the singular value decomposition of $K_m$ it holds that
		\begin{align*}
		\sigma_{jm}=\sqrt{\lambda_{jm}},\qquad
		u_{jm}=w_{jm},\qquad
		v_{jm}(\cdot)= \frac{1}{\sigma_{jm}}\sum_{l=1}^m(w_{jm})_l\kappa(\xi_{lm},\cdot),
	\end{align*}
	where $(w_{jm})_l$ is the $l$-th component of the vector $w_{jm}$.
	\end{proposition}
	Consequently, expressing the estimator $f_{k,m}^\delta$ in the basis $\{\kappa(\xi_{lm},\cdot), l=1,\dots,m\}$ and using Proposition \ref{prop0}, we get
	\begin{align}\label{disc:est}
	f_{k,m}^\delta=\sum_{l=1}^m\alpha_l \kappa(\xi_{lm},\cdot)\qquad\mbox{with}\qquad \alpha_l:=\alpha_l(k,m,\delta)=\sum_{j=1}^k\frac{(g_m^\delta,w_{jm})_{\R^m}}{\lambda_{jm}} (w_{jm})_l.
	\end{align}
	We note that if $\lambda_{jm}>0$ for all $j\le m$, then Proposition \ref{prop0} actually holds for any continuous kernel $\kappa$  and not just for our specific choice of \eqref{s1:e1a}. The first main contribution of this article is the derivation of a precise bound for the error $\|f_{k,m}^\delta-f\|$, in particular regarding its dependence on the initial measurement grid. This analysis relies critically on the fact that for our setting we have explicit representations of the singular value decompositions of both $K$ and $K_m$. Thus we deduce that the regularized approximation \eqref{s1:e4} to the solution of the continuous integral equation \eqref{s1:e0} from noisy point evaluations \eqref{s1:e2} can be constructed as follows: Determine the $k$ leading eigenvalues and vectors of the matrix $T_m$ from \eqref{s1:e2b} and use them to compute the coefficients of $f_{k,m}^\delta$ in the basis $\{\kappa(\xi_{lm},\cdot), l=1,\dots,m\}$ via \eqref{disc:est}. The concrete choice and the effect of the truncation index $k$ will be discussed below. First, we will analyze the accuracy of $f_{k,m}^\delta$ in \eqref{s1:e4}. Then, in the next section, we will present our new modified approach based on averaging, which achieves the same order of accuracy. The main advantage of our new modified approach is that it is based on the matrix $T_{m_o}$ in \eqref{s1:e2b} for $m_o$ less than $m$, which reduces the computational costs significantly. 

  We will now look at the accuracy of $f_{k,m}^\delta$ in more detail. First, our preliminary example with the kernel \eqref{s1:e1a} is particularly simple. {  Here we can not only compute the eigenvalue decomposition of the continuous operator $K^*K$ (see \cite{courant1989methods}, p.371/373), but also explicitly determine $T_m$ and the eigenvalue decomposition of $K_m^*K_m$}, which allows us to compute the error $\|f_{k,m}^\delta-f\|$ exactly. We take the following result from \cite{jahn2022cross}:
\begin{lemma}\label{lem0}
	The singular values and left singular vectors of $K$ from \eqref{s1:e0} are 
		\begin{equation*}
		\sigma_j=(\pi j)^{-2},\qquad v_j(x)={  \sqrt{2} }\sin(j\pi x), \quad j\in\N.
	\end{equation*}	
	The singular values and left singular vectors of $K_m$ are 
		\begin{equation*}
		\sigma_{jm}=\frac{1}{4(m+1)^{3/2}\sin^2\left(\frac{{  j\pi}}{2(m+1)}\right)}\sqrt{1-\frac{2}{3} \sin^2\left(\frac{{  j\pi}}{2(m+1)}\right)}
	\end{equation*}
		and 
		\begin{equation*}
		v_{jm}(x)={  \sqrt{\frac{2}{(m+1)\sigma_{jm}^2}}} \sum_{l=1}^m \sin\left(j \pi \xi_{lm}\right) \kappa(\xi_{lm},x)
	\end{equation*}
		with $j=1,...,m$.
\end{lemma}

 In order to derive explicit error bounds for $f_{k,m}^\delta$ for the given kernel \eqref{s1:e1a}, one must impose certain smoothness conditions on $f$. Usually one assumes that $f$ belongs to some unknown source set $\mathcal{X}_{s,\rho}$ of the form 
\begin{align*}
	\mathcal{X}_{s,\rho}:&=\left\{\varphi_s\left(K^*K\right)h~:~h\in L^2(0,1),~\|h\|\le\rho\right\}=\left\{ \sum_{j=1}^\infty \varphi_s\left(\sigma_j^2\right) (h,v_j) v_j~:~h\in L^2(0,1),~\|h\|\le\rho\right\}
\end{align*}
and $\varphi_s(t):=t^\frac{s}{2}$, where $s,\rho>0$ are unknown parameters. For functions in $\mathcal{X}_{s,\rho}$ we can indeed give quantitative estimates for the approximation error. We do this in Section \ref{s3:s1} and discuss their relation to Sobolev smoothness. To determine the estimator $f_{k,m}^\delta$, we must first choose the truncation index $k\le m$. The optimal choice for $k$ minimizes the expected distance to $f\in\mathcal{X}_{s,\rho}$ and gives the error
\begin{equation}\label{s1:e2c}
	{\rm err}(\delta,m,s,\rho):=\min_{k= 0,...,m} \sup_{f\in\mathcal{X}_{s,\rho}}\E\|f_{k,m}^\delta-f\|^2\asymp \left(\frac{\delta^2}{m}\right)^\frac{4s}{5+4s}(\rho^2)^\frac{5}{5+4s} +\Delta_m^{2},
\end{equation}
where $\Delta_m$ denotes the discretization error {$\|f-P_{\mathcal{N}(K_m)^\perp}f\|$, and $P_{\mathcal{N}(K_m)^\perp}:L^2(0,1)\to \mathcal{N}(K_m)^\perp$ is the projection onto the ansatz space, see \eqref{ansatz}}. Note here that the $k$ that minimizes \eqref{s1:e2c} depends on the smoothness parameters $s$ and $\rho$, which are usually unknown.  Finally, note that the first term in \eqref{s1:e2c} is {an} universal optimal error bound. This means that no reconstruction method can yield a smaller error over $\mathcal{X}_{s,\rho}$ uniformly, at least if one ignores the multiplicative constant. This was proved for example in \cite{mathe2017complexity}, where an asymptotically statistically equivalent functional white noise model (with variance $\delta^2/m$) is considered. In this model one starts from noisy linear functionals of the exact data instead of noisy point evaluations, which allows a discretization-free analysis. 

{ The remainder of the article is structured as follows: In Section \ref{s3:s1} we will introduce our new method and show that it attains the same error rate \eqref{s1:e2c}. Then, in Section \ref{s5}, the method is extended to general integral equations, discretized by means of a quadrature rule. In Section \ref{adap} we discuss an adaptive choice for determining the optimal $k$ using only the noisy point evaluations and the noise level $\delta$. The theoretical results are supported by numerical experiments in Section \ref{s6}. Finally, the formal proofs are given in the appendix.}
\section{Main results}\label{s3}
In this section we introduce our approach in more detail and state the main error bounds. We first consider the special case given in \eqref{s1:e1a} and then extend to general Fredholm equations.
\subsection{Results for the kernel with known spectral decomposition}\label{s3:s1}
For the kernel given in \eqref{s1:e1a}, we have the following simple relation between the abstract smoothness defined by the source set $\mathcal{X}_{s,\rho}$ and the classical smoothness, see Proposition 3.9 in \cite{jahn2022cross}: For $f\in\mathcal{X}_{s,\rho}$, if $s>\frac{3}{4}$, then $f$ is differentiable and if $s>\frac{5}{4}$, then $f$ is twice differentiable. Thus, to derive our main error estimates, we assume in this section that
\begin{equation}\label{s2:e1}
	f\in\mathcal{X}_{s,\rho} \quad \mbox{with}\quad s> 3/4.
\end{equation}
Note at this point that this is a substantial prerequisite for $f$, which allows us to derive simple bounds on the discretization error {quantifying} how well the unknown $f$ can be approximated in $\mathcal{N}(K_m)^\perp$. 

The main idea of our method is that the initial discretization given by the $m$ noisy point evaluations may be unnecessarily fine relative to the data noise $\delta$ and the unknown solution $f$. To this end, $\E\|f_{k,m}^\delta-f\|^2$ is split into three terms, first a data propagation error coming from the noise on the point evaluations, then an approximation error of the projected unknown solution, and finally a discretization error. While the first two terms depend, among other things, on the level of truncation $k$, the last term depends only on the parameter of the discretization dimension $m$. Roughly speaking, if for the minimizing $k$ the first two terms dominate the last term, then we see that the initial discretization based on the design matrix of size $m\times m$ was unnecessarily large. Let us be more precise: With $f_m:=\sum_{j=1}^m(f,v_{jm})v_{jm}$ denoting the projection of $f$ onto $\mathcal{N}(K_m)^\perp$, the error decomposition is 
\begin{align}\notag
	f_{k,m}^\delta-f&= \sum_{j=1}^k \frac{\left(g^\delta_m,u_{jm}\right)}{\sigma_{jm}} v_{jm} - \sum_{j=1}^k(f,v_{jm})v_{jm} - \sum_{j=k+1}^m(f,v_{jm})v_{jm} + f_m-f\\\label{err:decomp}
	&=\sum_{j=1}^k \frac{\left(g^\delta_m-K_m f,u_{jm}\right)}{\sigma_{jm}} v_{jm} - \sum_{j=k+1}^m(f,v_{jm})v_{jm}  + f_m-f,
\end{align}
since $(f,v_{jm})=(f,\sigma_{jm}^{-1}K_m^*u_{jm}) = \sigma_{jm}^{-1}(K_mf,u_{jm})$. From orthogonality we have
\begin{align}\label{s2:e11}
	\E[\|f_{k,m}^\delta-f\|^2] &= \delta^2\sum_{j=1}^k\frac{1}{\sigma_{jm}^2}  + \sum_{j=k+1}^m (f,v_{jm})^2 + \|f_m-f\|^2,
\end{align}
and with the source condition \eqref{s2:e1} and Lemma \ref{lem0} we obtain the explicit upper bound 
\begin{equation}\label{s3:e1a}
	\E[\|f_{k,m}^\delta-f\|^2] \le C \left(\frac{k^5}{m}\delta^2 + k^{-4s}\rho^2 + \frac{\|f'\|^2}{m^2}\right)
\end{equation} 
for the error, where $C>0$ is a constant. This will be shown in the proof of Theorem \ref{t2} below. It will {moreover} be shown there that essentially the right-hand side is also a lower bound in many cases, up to a constant factor. 
{  
\begin{remark}\label{rem:disc:err}
    Note that if $f$ is twice differentiable, which is the case for example when $s>\frac{5}{4}$, we can replace $\frac{\|f'\|^2}{m^2}$ by $\frac{\|f''\|^2}{m^4}$, see \cite{jahn2022cross}. An often used alternative decomposition of the error, applicable in general situations, see e.g. \cite{nair2007regularized}, is based on operator monotone functions \cite{mathe2006regularization}. Due to the Lipschitz continuity of $\kappa$, one can show that $\|(m+1)^{-1}(K_m^*K_m)-K^*K\|\le m^{-1}$. Then, for $f=(K^*K)^\frac{s}{2}h$, with $h\in L^2(0,1)$ and $\|h\|\le \rho$, $s\le 2$, the terms $-\sum_{j=k+1}^m(f,v_{jm})v_{jm} + f_m -f$ in \eqref{err:decomp} are replaced by
    \begin{align*}
        -\sum_{j=k+1}^m\left(\left(\frac{K_m^*K_m}{m+1}\right)^\frac{s}{2}h,v_{jm}\right)v_{jm} + (I-P_{{\rm span}(v_{1m},...,v_{km})})\left(\left(\frac{K_m^*K_m}{m+1}\right)^\frac{s}{2}h-(K^*K)^\frac{s}{2}h\right),
    \end{align*}
    with the corresponding (quadratic) upper bound $$2k^{-4s}\rho^2 + 2\left(\|(m+1)^{-1}K_m^*K_m-K^*K\|\right)^s\|h\|^2\le 2k^{-4s}\rho^2+m^{-s}\rho^2.$$ In our case, the discretization error is of order $m^{-2}\|f'\|^2$ for $\frac{3}{4}<s<1$, and of order $m^{-4}\|f''\|^2$ for $\frac{3}{4}<s\le 2$, which is clearly an improvement. 
\end{remark}
}
In our error bound \eqref{s3:e1a} we can identify the variance, also called the data propagation error, and the bias part, consisting of an approximation error and a discretization error. As $k$ increases, the first term increases while the second term decreases. As $m$ increases, the first and third terms decrease. 

As explained above, our approach is based on the observation that $k$ only affects the first two terms in the error decomposition. It is easy to see that the minimizing $k$ roughly balances the first two terms (and thus satisfies $k^{5+{ 4}s}\asymp \frac{m\rho^2}{\delta^2}$). If the contribution of the third term (the discretization error) is much smaller, then the initial discretization dimension $m$ (for the design matrix), i.e. the number of point evaluations, was unnecessarily high. Thus, there is hope to somehow reduce $m$ without spoiling the overall error rate. But directly reducing the parameter $m$ (which would mean simply discarding some of the measurements) increases the first and third terms (with the effect being more pronounced for the third term).

Instead, our proposed strategy is based on averaging certain components of the measured data, which can keep the size of the first two terms constant while decreasing $m$. The idea of averaging point evaluations from a fine grid to obtain (approximations of) point evaluations on a coarser grid is borrowed from the recent preprints \cite{jahn2022discretisation, jahn2023noise}, where a combination of a data-driven regularization method and an adaptive choice of the discretization dimension, called discretization-adaptive regularization \cite{jahn2021optimal, jahn2022probabilistic}, was numerically tested and performed well. In the present article we shed light on this observed phenomenon and give a rigorous justification.  The reasoning is that if not too many components are averaged, then due to the smoothness of the uncorrupted data $g$, a reduction of the stochastic noise on the point evaluations is obtained. Of course, this depends on the concrete relationship between $\delta,m$ and the unknown data $g$.

We explore this approach explicitly. To do so, we define for $o\in\N$, so that $m_o:=m/o\in\N$, the averaged data
\begin{align*}
	\overline{g}^\delta_{m_o} =\begin{pmatrix} \frac{\sum_{j=1}^{ o} \left(g^\delta_{m}\right)_{(i-1){ o}+j}}{{ o}}\end{pmatrix}_{i=1,...,m_o}\in\R^{m_o}.
\end{align*}
Similarly, we define our averaged spectral cut-off estimator as
\begin{align*}
	\overline{f}^\delta_{k,m_o}:= \sum_{j=1}^k \frac{\left(\overline{g}^\delta_{m_o},u_{jm_o}\right)_{\R^{m_o}}}{\sigma_{jm_o}}v_{jm_o} = \sum_{l=1}^{m_o} \overline{\alpha}_l \kappa(\xi_{lm_o},\cdot)
 \end{align*}
 with
 \begin{align*}
\overline{\alpha}_l:=\overline{\alpha}_l(k,m,\delta)=\sum_{j=1}^k\frac{(\overline{g}_{m_o}^\delta,w_{jm_o})_{\R^{m_o}}}{\lambda_{jm_o}} (w_{jm_o})_l.
\end{align*}
We briefly mention here that the calculation of the averaged data $
\overline{g}^\delta_{m_o}$ costs $m$ operations (regardless of the value of ${ o}$), since each entry of the initial data $g^\delta_m$ is touched once. This is cheap compared to the cost we face when we have to approximate $T_m$ and its eigenvalue decomposition for general kernels in the next section. We define the minimax error of the averaged estimator as 
\begin{equation*}
	\overline{\rm err}(\delta,m_o,s,\rho):=\min_{k={ 0},...,{ m_o}}\sup_{f\in\mathcal{X}_{s,\rho}}\E\|\overline{f}_{k,m_o}^\delta-f\|^2.
\end{equation*}
The following theorem tells us how much data we can average without spoiling the overall error rate.
\begin{theorem}\label{t2}
	For
 {  
 \begin{equation}\label{t2:def}
O_{m,\delta,\rho}:=\left\{o \in \N~:~o\le \max\left(\sqrt{\frac{(m+1)\delta^2}{\rho^2}},1\right),~\frac{m}{o}\in\N\right\}
 \end{equation}}
 and $m_o=m/o, o\in O_{m,\delta,\rho}$,  there holds
		\begin{equation}\label{t2e}
		c\cdot{\rm err}(\delta,m,{s},\rho)\le \overline{\rm err}(\delta,m_{o},s,\rho)\le C\cdot {\rm err}(\delta,m,s,\rho)
	\end{equation}
with constants
$$
c:=\frac{16}{\pi^{4s+4}\left(3\pi^4 + 1/2\right)}\qquad\mbox{and}\qquad C:=\frac{15 \pi^{4s+8}}{16}.
$$
 \end{theorem}
Clearly, one is in particular interested in determing the largest number in $O_{m,\delta,\rho}$, which we call $o_m$ for brevity. Note that finding this number is a problem in practice, since $\rho$ is typically unknown. As demonstrated and explained below in the numerical part in Section \ref{s6}, an alternative reasonable choice for $o_m$ would be {  to replace $\sqrt{\frac{(m+1)\delta^2}{\rho^2}}$ with $\left(\frac{(m+1)^2\delta^2}{\|g'\|^2}\right)^\frac{1}{3}$ in \eqref{t2:def}}. Here, to determine $\|g'\|$ is an inverse problem itself, but a fairly simple one.   
\subsection{Extension to general Fredholm integral equations}\label{s5}
We now explain how our approach can be applied to various Fredholm integral equations. Consider the general integral equation 
\begin{equation}\label{s3:e0}
	Kf=g,
\end{equation}
where    
\begin{equation}\label{s3:e1}
	g(x)=(Kf)(x)=\int_0^1\kappa(x,y)f(y)\mathrm{d}y. 
\end{equation}
The integral kernel $\kappa$ is assumed to be two times differentiable. The adjoint of $K$ is given as $(K^*g)(y) = \int_0^1 {\kappa}(x,y)g(x)\mathrm{d}x$. Again, discretization on a grid given by the points $\xi_{1m},...,\xi_{mm}$ yields a semi-discrete operator, which we normalize this time. We then have
\begin{align}\label{s3:e000}
	K_m:L^2(0,1)&\to\R^m\\\notag
	f &\mapsto \left( {m^{-\frac{1}{2}}}(Kf)(\xi_{lm})\right)_{l=1}^m = \frac{1}{\sqrt{m}}\left(\int \kappa(\xi_{lm},y) f(y)\mathrm{d}y\right)_{l=1}^m.
\end{align}
We emphasize that, in general, we now can no longer evaluate the integrals exactly. Instead, we use midpoint collocation as a quadrature rule, i.e. we have weights $q_{1m}=...=q_{mm}=m^{-1}$ and knots $\xi_{jm}:=\frac{2j-1}{2m}$, $j=1,...,m$. We define $A_m\in\R^{m \times m}$ by $\left(A_m\right)_{ij}=q_{jm}\kappa(\xi_{im},\xi_{jm}) = \frac{1}{m}\kappa(\xi_{im},\xi_{jm})$. This method is taken from the popular open source toolbox \cite{Hansen:2007} and is used in the numerical examples in the following section.\footnote{Clearly, the results can be generalized to more general quadrature rules. In that case the normalization of $K_m$ should be adapted.} 
The midpoint collocation method has degree of exactness one, which implies that for any $h\in\mathcal{C}^2([0,1])$, there holds
\begin{equation*}
	\left|\int h(x){\rm d}x - \frac{1}{m}\sum_{i=1}^mh(\xi_{im})\right|\le \frac{1}{24} \frac{\|h^{(2)}\|_\infty}{m^2}.
\end{equation*}
Therefore, we make the following assumption about the smoothness of the integral kernel $\kappa${  : 
\begin{assumption}\label{a1}
	$K_m$ has full rank and $\kappa\in\mathcal{C}^2([0,1]^2)$ with 
 \begin{equation}\label{a1e1}
 C_K:=\sup_{x,y}\max_{\substack{n,n'\le 2\\n+n'\le 2}}|\partial^n_x \partial^{n'}_y{\kappa}(x,y)|.
 \end{equation}
\end{assumption}}
Similarly to the design matrix $T_m$ from the previous section, we set
		\begin{equation}\label{designmat}
		\left(S_m\right)_{ij=1}^m:=\frac{1}{m}\left( \int \kappa(x,\xi_{im})\kappa(x,\xi_{jm}){\rm d}x\right)_{ij=1}^m\in\R^{m\times m}.
	\end{equation}
Then there holds
\begin{align*}
\left(A_m^TA_m\right)_{ij} - \left(S_m\right)_{ij} &= \frac{1}{m}\left( \frac{1}{m}\sum_{l=1}^m \kappa(\xi_{lm},\xi_{im})\kappa(\xi_{lm},\xi_{jm}) - \int \kappa(x,\xi_{im})\kappa(x,\xi_{{j}m}){\rm d}x\right)=\mathcal{O}\left(\frac{1}{m^3}\right).
\end{align*}
In light of the above results, we base our approximation on the discrete singular value decomposition of $A_m\in\R^{m\times m}$ (which is effectively the eigenvalue decomposition of $A_m^TA_m\approx S_m$), which we denote by $(\tilde{z}_{jm},\tilde{w}_{jm},\tilde{\sigma}_{jm})$. Note that unlike in the previous section, where we could give closed form expressions for the eigenvalues and eigenvectors of $T_m$, here we have to rely on numerical algorithms that only approximate the singular value decomposition of $A_m$. Since there are algorithms that approximate the singular value decomposition to machine precision, we will not distinguish between the exact singular value decomposition of $A_m$ and its numerically approximated counterpart. However, we will discuss the computational cost of the numerical approximation in detail below. { For $\tilde{\sigma}_{jm}>0$} we define 
\begin{equation}\label{sing}
	\tilde{v}_{jm}(\cdot):=\frac{1}{\tilde{\sigma}_{jm}\sqrt{m}}\sum_{l=1}^m (\tilde{w}_{jm})_l \kappa(\xi_{lm},\cdot)\in L^2(0,1),
\end{equation}
which approximates the singular function $v_{jm}$ of the semi-discrete operator \eqref{s3:e000} (for $j$ small enough). We mention here that, in contrast to the exact singular functions $v_{jm}$, the functions $\tilde{v}_{jm}$ form an orthonormal basis only approximately and for $j$ small, see Lemma \ref{lem002} below. We now use \eqref{sing} to denote our (discrete) spectral cutoff estimator for the unknown $f$ as 
\begin{equation}\label{est1}	\tilde{f}^\delta_{k,m}:=\sum_{j=1}^k\frac{(g^\delta_m,\tilde{w}_{jm})_{\R^m}}{\sqrt{m} \tilde{\sigma}_{jm}} \tilde{v}_{jm},
\end{equation}
{ under the condition that $\tilde{\sigma}_{jm}>0$ for $j=1,...,k$. We later show that this condition is fulfilled for $k$ sufficiently small compared to $m$}. Similar to the previous section, expressing the estimator $\tilde{f}_{k,m}^\delta$ in the basis $\{\kappa(\xi_{lm},\cdot),~l=1,...,m\}$ yields
\begin{align}\label{est1:disc}
 \tilde{f}^\delta_{k,m} = \sum_{l=1}^m \tilde{\alpha}_l\kappa(\xi_{lm},\cdot)\qquad \mbox{with}\qquad \tilde{\alpha}_l:={\tilde{\alpha}_l(k,m,\delta)}= \sum_{j=1}^k \frac{(g^\delta_m,\tilde{w}_{jm})}{m\tilde{\sigma}_{jm}^2} (\tilde{w}_{jm})_l.
\end{align}
Before analyzing the accuracy of the estimator $\tilde{f}^\delta_{k,m}$, we discuss its computational cost. If the $k$ leading singular values and vectors of $A_m$ can be derived explicitly, the $k$ quantities $(g_m^\delta,\tilde{w}_{jm})_{\R^m}/(m\tilde{\sigma}_{jm}^2)$, with $j=1,...,k,$ cost $2m+1$ operations each, and the $m$ coefficients $\tilde{\alpha}_l$, with $l=1,...,m,$ cost $m(2k-1)$ operations altogether. Next, we discuss the cost of computing the leading singular values and vectors of $A_m$ numerically. Highly stable and accurate methods for doing this transform the matrix $A_m$ first into a tridiagonal matrix, then compute the full singular value decomposition of this matrix, and finally keep only the leading $k$ values and vectors, see e.g. \cite{golub2013matrix}. These algorithms typically require $\mathcal{O}(m^3)$ operations and are therefore very expensive. Since we are only interested in the $k$ leading values and vectors, it is natural to consider methods that determine them directly. Typically, these methods are much cheaper, but also less accurate. We discuss some of them in the following.  If matrix-vector multiplications can be computed quickly, Krylov subspace methods such as the Lanczos or Arnoldi algorithm are commonly used. However, summarizing the computational costs with corresponding accuracy guarantees is challenging because these methods are inherently numerically unstable.  The specific costs depend heavily on the matrix properties and the effort spent to stabilize the routines. For more information, we refer to the survey article \cite{halko2011finding}. As a (usually overly optimistic) rule of thumb, the typical cost of such methods is proportional to $k C^{\rm mult}_m + m k^2$, where $C^{\rm mult}_m$ denotes the cost of an exact or at least approximate matrix-vector multiplication with the initial matrix $A_m$. Apart from the setting with sparse matrix $A_m$, fast multipole expansions \cite{greengard1987fast} or $H$ (and $H^2$) matrix arithmetic \cite{hackbusch1999sparse} can reduce $C^{\rm mul}_m$ to the order of $\log(m)m$ or even $m$. However, this requires the kernel to be asymptotically smooth. Also, as mentioned above, additional work is required to stabilize the resulting methods, and exact error guarantees are rare. Therefore, we focus our discussion on the use of provably stable two-step procedures to determine the truncated spectral decomposition, see \cite{halko2011finding} and \cite{kishore2017literature} for overviews. These procedures work as follows: First, an approximate basis of size $n<m$ for the range of $A_m$ is computed and used to determine a rank-$n$ approximation of $A_m$, which we call $A_m^{(n)}$. Second, the full singular value decomposition of this approximation $A_m^{(n)}$ is computed by established methods \cite{golub2013matrix} and the leading $k$ values and vectors are used to approximate $\tilde{\sigma}_{jm}$ and $\tilde{w}_{jm}$ for $j=1,\dots,k$. To do this, the auxiliary parameter $n$ must be chosen much larger than the target rank $k$ to ensure that the leading $k$ singular values and vectors of $A_m^{(n)}$ approximate the leading eigenvalues and vectors of $A_m$ with sufficient accuracy, see \eqref{comp:1} below\footnote{We emphasize that this is in sharp contrast to many classical settings where the goal is simply to find a near-optimal rank-$k$ approximation of $A_m$. Here typically $n=k+p$ where  $p$ is a small oversampling parameter.}. Several methods are available to perform the first step, e.g. based on rank-revealing QR decompositions \cite{gu1996efficient}, random projections \cite{rokhlin2010randomized}, or cross/skeleton approximations \cite{tyrtyshnikov2000incomplete, bebendorf2000approximation}. Using rank-revealing QR-pivoting or random projections for the first step requires roughly $nm^2$ operations. In the case of the cross/skeleton approximation, this can be reduced to $\mathcal{O}(nm)$, but it requires that the integral kernel be approximately smooth. The second step, the determination of the full singular value decomposition of $A_m^{(n)}$, costs $\mathcal{O}(n^2m)$ and thus the total cost is of the order of $\mathcal{O}\left(n C_m + m n^2\right)$, where $C_m$ depends on the approximation approach and the regularity property of the specific kernel. Overall, depending on the actual kernel $\kappa$, the computational cost for \eqref{disc:est} is somewhere between $n^2m$ and $nm^2$. This brings us to the choice of $n$. Let $\tilde{w}_{jm}^{\rm approx}$ and $\tilde{\sigma}_{jm}^{\rm approx}$ denote the $j$-th singular value and right singular vector of $A_m^{(n)}$, and let
 $\tilde{f}^{\delta,{\rm approx}}_{k,m}$ and $\tilde{\alpha}^{\rm approx}\in\R^m$ be defined as in \eqref{est1:disc}. Theorem 4.2 from \cite{ito2019regularized} shows that for $k$ not too large
\begin{equation}\label{comp:1}
\frac{\left\|\tilde{f}^{\delta,{\rm approx}}_{k,m}-\tilde{f}_{k,m}^\delta\right\|}{\left\|\tilde{f}_{k,m}^\delta\right\|} = \mathcal{O}\left(\frac{\|A_m-A_m^{(n)}\|}{\tilde{\sigma}_{km}^2}\right).
\end{equation}
Furthermore, for the two-step procedure, the total approximation error is usually about\footnote{Depending on the decay and gaps of the spectrum, in some cases the first factor can be reduced to logarithmic terms.} 
 \begin{equation}\label{comp:2}
 \|A_m-A_m^{(n)}\|=\mathcal{O}\left(\sqrt{n m}\tilde{\sigma}_{n+1m}\right).
 \end{equation}
  Thus the value $n$ must be chosen such that at least $\tilde{\sigma}_{n+1m}=\mathcal{O}\left( \tilde{\sigma}_{km}^2(mn)^{-\frac{1}{2}}\right)$.
For example, suppose\footnote{In contrast to the previous section, there is no substantial dependence on $m$, since $K_m$ and $A_m$ are normalized here.} $\tilde{\sigma}_{jm}^2 {=} j^{-q}$ with $q>1$. Approximating $\tilde{f}_{k,m}^\delta$ up to the relative error $\varepsilon$ requires $\tilde{\sigma}_{n+1m} \approx \varepsilon \tilde{\sigma}_{km}^2 (nm)^{-\frac{1}{2}}$, which implies $n\approx k^\frac{2q}{q-1} \varepsilon^{-\frac{2}{q-1}} m^\frac{1}{q-1}$. Therefore, between $k^\frac{4q}{q-1} \varepsilon^{-\frac{4}{q-1}} m^{1+\frac{2}{q-1}}$ and $ k^\frac{2q}{q-1} \varepsilon^{-\frac{2}{q-1}} m^{2+\frac{1}{q-1}}$ operations are necessary to determine $\tilde{f}_{k, m}^\delta$ up to the relative error $\varepsilon$.  
Note that the optimal choice for $k$ that minimizes the error $\|\tilde{f}_{k{,}m}^\delta-f\|$ will depend on $m$, the noise level $\delta$, and the regularity of the exact solution. In the following we will again investigate the impact of using averaging to deduce a modified spectral cut-off estimator with reduced computational cost. The cost reduction is due to the fact that the modified estimator is based on $A_{m_o}$ with $m_o< m$. This clearly improves the computational cost discussed in the above paragraph, since the averaging itself costs only $m$ operations. 

As a next main result, we compute the error of our estimator. For this purpose, the following auxiliary lemma {estimates the difference between the first} $m$ singular values $\sigma_j$ of the continuous operator $K$ {and the singular values} $\tilde{\sigma}_{jm}$ of the quadrature matrix $A_m$.
\begin{lemma}\label{lem00}
	Let $(\tilde{\sigma}_{jm})_{j\le m}$ be the singular values of $(A_m)_{ij\le m} = \frac{1}{m}\kappa(\xi_{im},\xi_{jm})$. Then, it holds that 
		\begin{equation*}
		\left|\sigma_j^2-\tilde{\sigma}_{jm}^2\right|\le \frac{C_K^2}{3 m^2}
	\end{equation*}
		for $j=1,...,m$.
	\end{lemma}
Before formulating the main results of this paper, we introduce two important auxiliary lemmata. The first relates approximate eigenvectors, i.e. vectors $v$ satisfying $\|K^*K v - \lambda v\|\approx 0$ for some $\lambda>0$, to exact ones.
\begin{lemma}\label{lem1}
	Let $K:\mathcal{X}\to\mathcal{X}$ be {  compact and positive semi-definite} with orthonormal eigenbasis $(v_i)_{i\in\N}$ and corresponding eigenvalues $(\lambda_i)_{i\in\N}$. Now suppose there are $v\in\mathcal{X}$ and {  $\lambda, \varepsilon>0$} such that $\|K v - \lambda v\|\le \varepsilon$. Then we have 
 \begin{equation}\label{lem1:e01}
 \min_{i\in\N}|\lambda_i-\lambda|\le \frac{\varepsilon}{\|v\|}.
 \end{equation}
 Furthermore, {  let $I:=\arg\min_{i\in\N}|\lambda_i-\lambda|$ and define $c:=\min_{i\not\in I}|\lambda_i-\lambda|$. Denote by $Pv=P_{{\rm span}\{v_i~:~i\in I\}}v$ the orthogonal projection of $v$ onto the span of $\{v_i~:~i\in I\}$. Then either {$I=\N$} and $Pv=v$, or $c>0$ and we have 
 \begin{equation}\label{lem1:e02}
 \|v\|^2\ge (v,Pv)\ge \|v\|^2 - \frac{\varepsilon^2}{c^2}.
 \end{equation}}
\end{lemma}
Note that due to compactness the identity $I=\N$ holds only if $K=0$, in which case $c$ is not properly defined since $\N \setminus I$ is empty then. The next lemma shows that the constructed functions $\tilde{v}_{jm}$ in \eqref{sing}, based on the (discrete) singular vectors of $A_m$, are indeed approximate eigenfunctions of the continuous operator $K^*K$, if $j$ is not too large. {In order to handle a possible multiplicativity of the eigenvalues of $K$, we introduce the functions $\psi_\pm:\N\to\N$ defined as $\psi_-(j) = \min\{i\ge 1~:~\sigma_i=\sigma_j$\} and $\psi_+(j) = \max\{i\ge 1~:~\sigma_i=\sigma_j\}$ and define
\begin{equation}\label{lem2:e1}
c_j:=\min\left(\sigma_{\psi_+(j)}^2-\sigma_{\psi_+(j)+1}^2, \sigma_{\psi_-(j)-1}^2-\sigma_{\psi_-(j)}^2\right),
\end{equation}
with the convention that $\sigma_0=\infty$. We define  \begin{equation}\label{lem2:mult}
 M_i:=1+\max_{j\le i}\left(\psi_+(j)-\psi_-(j)\right)
 \end{equation}
as the maximal multiplicity of the singular values up to index $i$, and define
 \begin{equation}\label{jm}
	J_m:=\max\left\{j\ge 1~:~1{ >\max\left(\frac{2C_K^2}{3c_im^2},\frac{10M_iC_K^{3}}{c_i \sigma_im^2}\right)}~\mbox{for all }~i\le j\right\},
\end{equation}
 as the index up to which the approximation is valid {(note that the second argument in $\max$ is usually dominating, unless $\sigma_i\gg 1$). It is} $J_m=\psi_+(j')$ for some $j'$, { and clearly, since $c_j\le \sigma_j^2-\sigma_{\Psi_+(j)+1}^2< \sigma_j^2$ and} by Lemma \ref{lem00},  for $j\le J_m$, we have}
\begin{equation*}
	\tilde{\sigma}_{jm}^2\ge \sigma_j^2 - |\tilde{\sigma}_{jm}^2-\sigma_j^2|\ge \sigma_j^2 - \frac{C_K^2}{3m^2}>\sigma_j^2-\frac{c_j}{2}> \frac{\sigma_j^2}{2}.
\end{equation*}
We also have the following result:
\begin{lemma}\label{lem002}
	Recall that $(A_m)_{ij}= \frac{1}{m} \kappa(\xi_{im},\xi_{jm})$ and $\tilde{v}_{jm}: =\frac{1}{\tilde{\sigma}_{jm}\sqrt{m}}\sum_{l=1}^m (\tilde{w}_{jm})_l \kappa(\xi_{lm}, \cdot)$, where $(\tilde{z}_{jm},\tilde{w}_{jm}{,\tilde{\sigma}_{jm}})$ is the singular value decomposition of $A_m$. Then, { for all $i,j\le J_m$,
		\begin{equation}\label{lem002:e1}
  |(\tilde{v}_{jm}, \tilde{v}_{im})-\delta_{ij}| \le \frac{C_K^2}{3\sigma_j\sigma_i m^2}
		\end{equation}
				and, for $P\tilde{v}_{jm}=P_{{\rm span}\{v_{\psi_-(j)},...,v_{\psi_+(j)}\}}\tilde{v}_{jm}$, i.e. the projection of $\tilde{v}_{jm}$ onto the span of $v_{\psi_-(j)},...,v_{\psi_+(j)}$, we have
				 \begin{equation}\label{lem002:e2}
			\|\tilde{v}_{jm}\|^2 \ge (\tilde{v}_{jm},P\tilde{v}_{jm}) \ge \|\tilde{v}_{jm}\|^2 - \frac{C_K^6}{c_j^2\sigma_j^2 m^4}. 
		\end{equation}}
\end{lemma}
We are now ready to give a bound on the total error. Compared to the specific setting considered in Section \ref{s3:s1} the estimates become more delicate. This is due to the following reasons: First, we allow multiple eigenvalues, second, the approximated singular functions \eqref{sing} are only approximately orthonormal, and third, we can only control the approximate singular functions corresponding to low frequency eigenvalues (relative to the noise level $\delta^2$). Thus we decompose the error slightly differently:	\begin{align}\notag
		\tilde{f}^\delta_{k,m}-f
		: =&\sum_{j=1}^k\frac{(g^\delta_m,\tilde{w}_{jm})_{\R^m}}{\sqrt{m}\tilde{\sigma}_{jm}}\tilde{v}_{jm} - f= \sum_{j=1}^k\frac{(g^\delta_m-g_m,\tilde{w}_{jm})_{\R^m}}{\sqrt{m}\tilde{\sigma}_{jm}}\tilde{v}_{jm}+  \sum_{j=1}^k(f,\tilde{v}_{jm}-P\tilde{v}_{jm})\tilde{v}_{jm}\\\label{s4e0}   &\qquad+  \sum_{j=1}^k (f,P\tilde{v}_{jm}) (\tilde{v}_{jm}-P\tilde{v}_{jm}) +\sum_{j=1}^k(f,P\tilde{v}_{jm})P\tilde{v}_{jm}- f.
	\end{align}
 We bound the different terms in the following theorem.
\begin{theorem}\label{t4}
	{  Under Assumption \ref{a1}, 	for all $k\le J_m$ it holds that 	
 \begin{align*}
		&\sqrt{\E\|\tilde{f}_{k,m}^\delta-f\|^2}\le \frac{2\delta}{\sqrt{m}}\sqrt{\sum_{j=1}^k\frac{1}{\sigma_j^2}} +\sqrt{\sum_{j=\psi_+(k)+1}^\infty (f,v_j)^2}+ 2 M_k \sqrt{\sum_{j=\psi_-(k)}^{\psi_+(k)}(f,v_j)^2}\\
  &\qquad+ \frac{(1+\sqrt{{2}M_k})C_K^{{3}}\|f\|}{m^2}\sqrt{\sum_{j=1}^k\frac{1}{c_j^2 \sigma_j^2}} + \frac{\sqrt{2} C_K^{{ 4}} \|f\|}{\sqrt{3}m^3}\sum_{j=1}^k\frac{1}{c_j\sigma_j^2}+ { M_k}\max_{i\le \psi_+(k)}\frac{20 M_{i} C_K^{{3}}\|f\|}{c_{i}\sigma_{i} m^2}.
	\end{align*}
} 
\end{theorem}
	Let us comment on this result: We note that the structure of the error bound in Theorem \ref{t4} is different from that in Theorem \ref{t2}. In the latter case, by introducing a source condition on $f$ (of sufficiently high degree), the special setting allowed us to analyze the term $\sum_{j=k+1}^m(f, v_{jm})v_{jm}$ (since we explicitly derived the singular functions of the semi-discrete operator), as well as the term $\sum_{j=1}^m(f,v_{jm})v_{jm}-f$ (due to special properties of the approximation space spanned by the anchored kernel functions). In the general setting of Theorem \ref{t4} this is not possible, and we base our analysis directly on the approximation of the singular functions of the continuous operator by those of the semi-discrete operator. Since only low-frequency singular vectors can be approximated with sufficient accuracy by the quadrature rule, this gives the constraint $k\le J_m$. Note that the third term can be replaced with zero in case that there is a spectral gap after $\sigma_k$, i.e. for $\psi_+(k)=k$.
 We now have the error for our estimator $\tilde{f}_{k,m}^\delta$ from \eqref{est1} based on the $m$ initial measurements $g^\delta_m$ and the quadrature matrix $A_m\in\R^{m\times m}$. As in the previous section, we identify a variance term, two approximation terms, and several terms due to discretization. Compared to Theorem~\ref{t2}, however, the situation here is more delicate, since the discretization error now depends on the truncation level $k$. This is due to the fact that now we have taken the error from the quadrature rule into account. Nevertheless, we may still encounter the situation where the initial discretization was too fine, in the sense that if we consider the $k$ such that the (square root of the) variance $\frac{2\delta}{\sqrt{m}}\sqrt{\sum_{j=1}^k\frac{1}{\sigma_j^2}}$ is balanced with the approximation error 
$$
  \sqrt{\sum_{j=\psi_+(k)+1}^\infty(f, v_{j})^2}+2 M_k \sqrt{\sum_{j=\psi_-(k)}^{\psi_+(k)}(f,v_j)^2},$$
  this $k$ can lead to a discretization error $$\frac{(1+\sqrt{{2}M_k})C_K^{{3}}\|f\|}{m^2}\sqrt{\sum_{j=1}^k\frac{1}{c_j^2 \sigma_j^2}} + \frac{\sqrt{2} C_K^{{ 4}} \|f\|}{\sqrt{3}m^3}\sum_{j=1}^k\frac{1}{c_j\sigma_j^2}+ { M_k}\max_{i\le \psi_+(k)}\frac{20 M_{i} C_K^{{3}}\|f\|}{c_{i}\sigma_{i} m^2}$$ of smaller order\footnote{Note here that the restriction $k \le J_{m}$ might hinder the optimal (balancing) choice of $k$, in particular if $M_k$ grows strongly.}. Therefore, we again average the $m$ initial measurements to obtain $m_o$ new measurements, and base the estimator on the singular value decomposition $( \tilde{z}_{jm_o},\tilde{w}_{jm_o}{,\tilde{\sigma}_{jm_o}})$ of $A_{m_o}\in\R^{m_o\times m_o}$ and the $\tilde{v}_{jm_o}$ from \eqref{sing}. So we consider the estimator
\begin{align}\label{t6e0}	\overline{\tilde{f}}^\delta_{k,m_o}:=\sum_{l=1}^{m_o} \overline{\tilde{\alpha}}_l \kappa(\xi_{lm_o},\cdot)\qquad\mbox{where}\qquad\overline{\tilde{\alpha}}_l:={ \overline{\tilde{\alpha}}_l(k,m_o,\delta)}= \sum_{j=1}^k \frac{(\overline{g}^\delta_{m_o},\tilde{w}_{jm_o})_{\mathbb{R}^{m_o}}}{m_o\tilde{\sigma}_{jm_o}^2} (\tilde{w}_{jm_o})_l,
\end{align}
with averaged data
\begin{align*}
	\overline{g}^\delta_{m_o}:=\left(\frac{1}{o}\sum_{i=1}^o \left(g^\delta_m\right)_{o(j-1)+i}\right)_{j=1}^{m_o}\in\R^{m_o}.
\end{align*}
As will be seen in the proofs, the main difference in the error analysis compared to the unaveraged estimator $\tilde{f}_{k,m}$ is the systematic data error introduced by averaging. In the special case with kernel \eqref{s1:e1a} treated in Section \ref{s3:s1}, this error was roughly $\|g'\|/m$. Now we will obtain a faster decay of this systematic component due to the additional smoothness assumptions for the kernel $\kappa$ and the special geometry of the discretization grid. We have the following result:
	{ 
\begin{theorem}\label{t6}
 For  $o\in\N$ and $m_o=\frac{m}{o}\in\N$, under Assumption \ref{a1}, for all $k\le J_{m_{ o}}$ it holds that
		\begin{align*}
		\sqrt{\E\left\| \overline{\tilde{f}}^\delta_{k,m_o}-f\right\|^2} &\le\frac{2\delta}{\sqrt{m}}\sqrt{\sum_{j=1}^k\frac{1}{\sigma_j^2}} + \frac{\|g''\|_\infty}{12m_o^2 \sigma_k}  + \frac{C_K\|g''\|_\infty}{12\sqrt{6}m_o^3}\sqrt{\sum_{j=1}^k\frac{1}{\sigma_j^4}} + \sqrt{\sum_{j=\psi_+(k)+1}^\infty(f,v_j)^2}\\
    &\quad+  2M_k \sqrt{\sum_{j=\psi_-(k)}^{\psi_+(k)} (f,v_j)^2} +\frac{(1+\sqrt{{2}M_k})C_K^{3}\|f\|}{m_o^2}\sqrt{\sum_{j=1}^k\frac{1}{c_j^2 \sigma_j^2}} + \frac{\sqrt{2} C_K^{4} \|f\|}{\sqrt{3}m_o^3}\sum_{j=1}^k\frac{1}{c_j\sigma_j^2}\\
    &\quad+ { M_k}\max_{i\le \psi_+(k)}\frac{20 M_{i} C_K^{3}\|f\|}{c_{i}\sigma_{i} m_o^2}.	
	\end{align*}
	\end{theorem}
As above, the fifth term can be omitted in case that $\psi_+(k)=k$.}  
We compare our upper bound in Theorem \ref{t6} for the averaged estimator with the upper bound in Theorem \ref{t4} for the non-averaged one. First, we observe that for $m_o,k\to \infty$ there holds 
\begin{equation*}
	\frac{\|g''\|_\infty}{12m_o^2 \sigma_k}  + \frac{C_K\|g''\|_\infty}{12\sqrt{6}m_o^3}\sqrt{\sum_{j=1}^k\frac{1}{\sigma_j^4}} \ll \frac{1}{m_o^2}\sqrt{\sum_{j=1}^k\frac{1}{c_j^2 \sigma_j^2}}.
\end{equation*}
Consequently, if the initial discretization was too fine, i.e, 
\begin{align*}
	&\frac{2\delta}{\sqrt{m}}\sqrt{\sum_{j=1}^k\frac{1}{\sigma_j^2}} + {  \sqrt{\sum_{j=\psi_+(k)+1}^\infty(f,v_j)^2}+  2M_k \sqrt{\sum_{j=\psi_-(k)}^{\psi_+(k)} (f,v_j)^2}}\\
 &{ \qquad\gg  \frac{(1+\sqrt{2M_k})C_K^3\|f\|}{m^2}\sqrt{\sum_{j=1}^k\frac{1}{c_j^2 \sigma_j^2}} + \frac{\sqrt{2} C_K^4 \|f\|}{\sqrt{3}m^3}\sum_{j=1}^k\frac{1}{c_j\sigma_j^2}+ M_k\max_{i\le {\psi_+(k)}}\frac{20 M_{i} C_K^3\|f\|}{c_{i}\sigma_{i} m^2}},   	
\end{align*}
and if $m_o$ is not too small, we see that the upper bounds in Theorem \ref{t4} and Theorem \ref{t6} are asymptotically the same. In principle, it should even be possible to obtain a stronger result here, comparable to Theorem \ref{t2}. However, we leave this as future work, since clearly some additional restrictions are needed to obtain an exact lower bound. Looking again at the example where $\tilde{\sigma}_{jm}^2 {=} j^{-q}$, we see that approximating $\tilde{f}^\delta_{k, m}$ to a relative error of $\varepsilon>0$ is more expensive by a factor between $o^{1+\frac{2}{q-1}}$ and $o^{2+\frac{1}{q-1}}$ than approximating $\overline{\tilde{f}}^\delta_{k,m_o}$ to the same relative accuracy.
\subsection{Adaptivity}\label{adap}
Finally, we consider the question of adaptivity, which concerns the concrete choice of the truncation level $k$ and the discretization dimension $m_o$. Since properties of the unknown solution $f$, such as the particular smoothness $\varphi_s$ or the norm $\rho$ of the source element, are usually unknown, strategies are needed that depend only on the measurements $g^\delta_m$ and possibly on the noise level $\delta>0$. Obviously, one is interested in adaptively finding $k$ and $m_o$ such that the error $\|\overline{f}^\delta_{k, m_o}-f\|$ reaches the infeasible optimal choice $\arg\min_{k}\|f^\delta_{k, m}-f\|$ by at least a multiplicative factor, analogously for $\overline{\tilde{f}}^\delta_{k,m_o}$ and $\tilde{f}^\delta_{k,m}$. A popular method to determine a regularization parameter in a data-driven manner is the discrepancy principle. This principle follows the paradigm that the regularization parameter should be chosen such that the residual norm of the candidate approximation is approximately of the same size as the data error. In contrast to classical settings where the discretization dimension is fixed and usually only one regularization parameter has to be chosen, here we need a strategy to determine both the truncation level $k$ and the dimension $m_o$. In the following, we propose a modified discrepancy principle, which can be seen as a multiscale method. For simplicity, we restrict ourselves to the case where $m=a^n$ for some $a,n\in\N$. The expected squared data error is
\begin{align}\notag
	&\E\left[\|\overline{g}^\delta_{m_o}-g_{m_o}\|_{\R^{m_o}}^2\right]= \E\left[\left\| \overline{g}^\delta_{m_o} - \overline{g}_{m_o}\right\|_{\R^{m_o}}^2\right] + \left\|\overline{g}_{m_o}-g_{m_o}\right\|_{\R^{m_o}}^2\\\label{adap:e1}
 &~= \delta^2\E\left\|\begin{pmatrix} \bar{Z}_1 \\ \vdots\\ \bar{Z}_{m_o} \end{pmatrix}\right\|_{\R^{m_o}}^2 +\left\|\overline{g}_{m_o} - g_{m_o}\right\|^2_{\R^{m_o}}
	= \frac{m_o}{o}\delta^2 +\left\|\begin{pmatrix} \frac{\sum_{t=1}^{o} g\left(\xi_{tm}\right)}{o}-g(\xi_{1m_o}) \\ \vdots \\ \frac{\sum_{t=1}^{o}g(\xi_{{((m_o-1)o+t)}m})}{o} - g(\xi_{m_om_o}) \end{pmatrix}\right\|^2_{\R^{m_o}},
\end{align}
and the residual norm is $\|K_{m_{ o}}\overline{f}_{k,m_o} - \overline{g}^\delta_{m_o}\|_{\R^{m_o}}$. First, note that for $\kappa(x,y)=\min(x(1-y),y(1-x))$ and a uniform grid, we can compute the residual norm exactly as
\begin{equation*}
	\|K_{m_o} \overline{f}_{k,m_o}^\delta - \overline{g}_{m_o}^\delta\|_{\R^{{m_o}}}^2 = \sum_{j=k+1}^{m_o}(\overline{g}^\delta_{m_o},u_{jm_o})_{\R^{m_o}}^2.
\end{equation*}
It is monotonically decreasing in $k$ and is zero for $k=m_o$. The systematic error (the second term in \eqref{adap:e1}) can be estimated as in \eqref{s3e01} below, and we get 
\begin{align*}
	\|\overline{g}_{m_o}-g_{m_o}\|_{\R^{m_o}}^2\le \frac{o}{{m+1}}\|g'\|^2.
\end{align*}
Consequently, for $\kappa(x,y)=\min(x(1-y),y(1-x))$ discretized on a uniform grid, and fixed dimension $m_o$ and  parameter $\tau>1$, we define the truncation level determined by the discrepancy principle as 
\begin{equation}
	k_{\rm dp}^\delta(m_o):=\min\left\{k\ge 0 ~:~\|K_{m_o}\overline{f}_{k,m_o}^\delta - \overline{g}^\delta_{m_o}\|_{\R^{m_o}}\le\tau \sqrt{\frac{\|g'\|^2}{m_o}+\frac{m_o}{o}\delta^2}\right\}.
\end{equation}
This choice is intended to balance the residual norm and the (expected) data error. In the case of a general kernel $\kappa$ (fulfilling Assumption \ref{a1}) and the discretization by midpoint collocation we assume that $\dim\mathcal{R}\left(A_{m_o}\right) = m_o$, which implies that $(\tilde{w}_{jm_o})_j$ forms an orthonormal basis of $\R^{m_o}$. Then we have
\begin{align*}
	\left(K_{m_o} \overline{\tilde{f}}_{k,m_o}^\delta\right)_i &=\sum_{j=1}^k\frac{(\overline{g}^\delta_{m_o},\tilde{w}_{jm_o})_{\R^{m_o}}}{\sqrt{m_o}\tilde{\sigma}_{jm_o}} \left(K_{m_o} \tilde{v}_{jm_o}\right)_i\\
	&=\sum_{j=1}^k \frac{(\overline{g}^\delta_{m_o},\tilde{w}_{jm_o})_{\R^{m_o}}}{\sqrt{m_o}\tilde{\sigma}_{jm_o}}  \frac{1}{\sqrt{m_o}\tilde{\sigma}_{jm_o}}\sum_{t=1}^{m_o}(\tilde{w}_{jm_o})_t \int \kappa(\xi_{tm_o},y)\kappa(\xi_{im_o},y){\rm d}y 
\end{align*}
and if we also estimate the integral by the midpoint collocation quadrature, we get
\begin{align*}
	&\sum_{j=1}^k \frac{(\overline{g}^\delta_{m_o},\tilde{w}_{jm_o})_{\R^{m_o}}}{\sqrt{m_o}\tilde{\sigma}_{jm_o}}  \frac{1}{\sqrt{m_o}\tilde{\sigma}_{jm_o}}\sum_{t=1}^{m_o}(\tilde{w}_{jm_o})_t \frac{1}{m_o}\sum_{s=1}^{m_o} \kappa(\xi_{tm_o},\xi_{sm_o})\kappa(\xi_{im_o},\xi_{sm_o})\\
	&~~=~\sum_{j=1}^k \frac{(\overline{g}^\delta_{m_o},\tilde{w}_{jm_o})_{\R^{m_o}}}{\tilde{\sigma}_{jm_o}}  \frac{1}{m_o}\sum_{s=1}^{m_o}(\tilde{z}_{jm_o})_s \kappa(\xi_{im_o},\xi_{sm_o})=\sum_{j=1}^k (\overline{g}^\delta_{m_o},\tilde{w}_{jm_o})_{\R^{m_o}}  (\tilde{w}_{jm_o})_i,
\end{align*}
with quadrature error
\begin{align*}
	&\left|(K_{m_o} \overline{\tilde{f}}^\delta_{k,m_o})_i - \sum_{j=1}^k (\overline{g}^\delta_{m_o},\tilde{w}_{jm_o})_{\R^{m_o}}  (\tilde{w}_{jm_o})_i,\right|\\
	&~\le\left|\sum_{j=1}^k \frac{(\overline{g}^\delta_{m_o},\tilde{w}_{jm_o})_{\R^{m_o}}}{m_o\tilde{\sigma}_{jm_o}^2}  \sum_{t=1}^{m_o}(\tilde{w}_{jm_o})_t \left[\int \kappa(\xi_{tm_o},y)\kappa(\xi_{im_o},y){\rm d}y - \frac{1}{m_o}\sum_{s=1}^{m_o} \kappa(\xi_{tm_o},\xi_{sm_o})\kappa(\xi_{im_o},\xi_{sm_o})\right]\right|\\
	&~\le~\frac{C_K^2}{6m_o^2}\left|\sum_{j=1}^k\frac{(\overline{g}^\delta_{m_o},\tilde{w}_{jm_o})_{\R^{m_o}}}{m_o\tilde{\sigma}_{jm_o}^2} \sum_{t=1}^{m_o}(\tilde{w}_{jm_o})_t\right|\le \frac{C_K^2}{6 m_o^3 \tilde{\sigma}_{km_o}^2} \left|\sum_{j=1}^k(\overline{g}_{m_o}^{ \delta},\tilde{w}_{jm_o})_{\R^{m_o}}\right| \left|\sum_{t=1}^{m_o}(\tilde{w}_{jm_o})_t\right|\\
	&~\le~{\frac{C_K^2 \sqrt{k}\|\overline{g}_{m_o}^\delta\|_{\R^{m_o}}}{3 m_o^\frac{5}{2} \sigma_{k}^2}.}
\end{align*}
Now, since $(\tilde{w}_{jm_o})_{j=1}^{m_o}$ is an orthonormal basis in $\R^{m_o}$, we can estimate the residual norm as
\begin{align}\notag
	&\left\| K_{m_o} \overline{\tilde{f}}^\delta_{k,m_o} - \overline{g}_{m_o}^\delta\right\|_{\R^{m_o}}\\\notag
	&~~{\le} \left\| \sum_{j=1}^k(\overline{g}^\delta_{m_o},\tilde{w}_{jm_o})_{\R^{m_o}} \tilde{w}_{jm_o} - \overline{g}_{m_o}^{ \delta}\right\|_{\R^{m_o}} + \sqrt{\sum_{i=1}^{m_o}\left((K_{m_o} \overline{f}^\delta_{k,m_o})_i - \sum_{j=1}^k (\overline{g}^\delta_{m_o},\tilde{w}_{jm_o})_{\R^{m_o}}  (\tilde{w}_{jm_o})_i\right)^2}\\\label{eq:resi}
	&~~\le \sqrt{\sum_{j=k+1}^{m_o}(\overline{g}^\delta_{m_o},\tilde{w}_{jm_o})_{\R^{m_o}}^2} + {\frac{C_K^2\sqrt{k}\|\overline{g}^\delta_{m_o}\|_{\R^{m_o}}}{3\sigma_k^2m_o^2}.}
\end{align}
Finally, we define the discrepancy principle using
\begin{equation}\label{dp}
	k_{\rm dp}^\delta(m_o):=\min\left\{ k \ge 0~:~ \sqrt{\sum_{j=k+1}^{m_o}(\overline{g}^\delta_{m_o},\tilde{w}_{jm_o})_{\R^{m_o}}^2} \le\tau \sqrt{{\rm err}_{\rm sys}^2(m_o) + \frac{m_o}{o}\delta^2}\right\},
\end{equation}
where we estimate the systematic error either as in Section \ref{s3:s1} for the special kernel \eqref{s1:e1a} by
\begin{equation*}
	{\rm err}_{\rm sys}^2(m_o):=\frac{\|g'\|^2}{m_o},	
\end{equation*}
or as in Section \ref{s5} for the general kernel and the discretization grid $\xi_{jm}:= \frac{2j-1}{2m}, j=1,...,m,$ by
\begin{equation*}
	{\rm err}_{\rm sys}^2(m_o):=\frac{\|g''\|_\infty^2}{9\cdot 64 m_o^3},
\end{equation*}
which is usually a much better bound, see \eqref{t6e1} below. Note that the discrepancy principle is well-defined, even though, by \eqref{eq:resi}, we approximate the residual norm accurately only for $k$ not too large. In any case, we emphasize that knowledge of $\delta$ and either $\|g'\|$ or $\|g''\|_\infty$ is required here. Usually these two quantities can be estimated from the initial data $g^\delta_m$. Note that the determination of $\|g'\|$ or $\|g''\|_\infty$ is itself an inverse problem, but a rather mild one. Obviously, it would be of interest to derive a method that does not require this additional knowledge. We will comment on this issue in Section \ref{s7}. Our adaptive method iteratively applies the discrepancy principle on a scale of different discretization levels $m_{{o}_0}\le m_{{o}_1} \le ... \le m$, starting with a very coarse discretization, i.e. some $m_{{o}_0} =a^{n_0}\le m$. For each discretization level $m_o$ the classical discrepancy principle \eqref{dp} is applied.  The method stops when the truncation index determined by the discrepancy principle decreases for the first time. See Algorithm 1 below for the numerical implementation. 

In the following section, we will apply the discrepancy principle to some examples and compare its performance to that of the infeasible optimal choice $k_{\rm opt}^\delta(m)=\arg\min_{k\le m}\|f^\delta_{k,m}-f\|$. We indeed observe that the discretization dimension where the discrepancy principle is maximal yields approximately the optimal error.
\section{Numerical Experiments}\label{s6}
In this section we present the results of our numerical experiments.  The initial number of point evaluations is $m=4^{6}=4096$ and for the reduced data we use $m_o\in\{4^{6},...,4^2\}$. We perturb the evaluations with i.i.d. Gaussian noise, but note that the results hold for any other centered i.i.d. white noise with finite variance.  We choose the noise level $\delta$ implicitly via the signal-to-noise ratio (SNR), which is defined as 
\begin{equation*}
	{\rm SNR}:=\frac{\|g\|}{\sqrt{\E\left[\|g^\delta_m-g\|^2\right]}} = \frac{\|g\|}{\sqrt{m}\delta}.
\end{equation*}
We will calculate the following quantities:
\begin{itemize}
	\item The optimal truncation level $k_{\rm opt}:=\arg\min_{k\le m_{ o}}\|\overline{f}^\delta_{k,m_o} - f\|$ with the corresponding optimal (relative) error $e_{\rm opt}:=\min_{k\le m_{o}}\|\overline{f}^\delta_{k,m_o}-f\|/\|f\|$ (Section \ref{s3:s1} and \ref{s5}).
	\item The data-driven truncation level $k_{\rm dp}$ determined by Algorithm 1 with the corresponding (relative) error $e_{\rm dp}:=\|\overline{f}^\delta_{k_{\rm dp},m_o}-f\|/\|f\|$, with parameter $\tau=1.5$ in \eqref{dp} (Section \ref{adap}),
\end{itemize}
and similarly for  $\overline{\tilde{f}}_{k,m_o}^\delta$. Here we take the average of $50$ independent runs. We start with the integral equation \eqref{s1:e1} with kernel $\kappa(x,y)=\min(x(1-y),y(1-x))$ from Section \ref{s3:s1}, for which we know the exact singular value decomposition for both the continuous and the discretized case. In the following we will denote it by {\tt deriv2} (because of its relation to the Poisson equation). We define the exact unknown solution over the singular vectors and set $f_i:=\sum_{j=1}^D \varphi_{s_i}(\sigma_j^2) v_j$ with  $D=5000$ and smoothness parameter $s_{i}=\frac{1+2i}{8}$, $i=0,1,2$, and $\varphi_s(t)=t^\frac{s}{2}$. The noise is varied by the signal-to-noise ratio ${\rm SNR}\in\{8^3,8^2,8,1\}$. Note that here all errors are computed exactly, for the formulas we refer to \cite{jahn2022cross}. In view of Theorem \ref{t2} it would be interesting to determine the optimal discretization level { $m_{o_m}$ from \eqref{t2:def}}, since for this discretization level we expect roughly the same optimal error as for the initial dimension $m$. However, even with simulated data, $m_{o_{{m}}}$ is difficult to obtain because the source element $h$ (and hence its norm $\rho$) for $f$ is not unique. As an alternative, we compute here another reasonable a priori choice for the optimal discretization dimension, which arises naturally in particular when considering the discrepancy principle. We define it implicitly by balancing  the two contributions of the data error, i.e., the variance due to random noise on the measurements and the bias due to averaging, in
\begin{equation*} 
	\E\|\overline{g}^\delta_{m_o}-g_{m_o}\|^2 \le \frac{\|g'\|^2}{m_o}+\frac{m_o}{o}\delta^2.
\end{equation*} 
This way we get 
\begin{equation}\label{apriorimo}
	\frac{\|g'\|^2}{m_o}\stackrel{!}{=}\frac{m_o}{o}\delta^2\quad \Longrightarrow\quad o=\left(\frac{m^2\delta^2}{\|g'\|^2}\right)^\frac{1}{3},
\end{equation}
    and we set $m_o^{\rm apriori}$ as the element in $\{4^6,...,4^2\}$ which is closest to $m/o$, where $o$  is determined through \eqref{apriorimo}. In the following tables, we print the corresponding column number belonging to $m^{\rm apriori}_o$ in bold.
The results are shown in Tables 1 and 2 below. Each table consists of four blocks corresponding to the four different SNRs. The five columns are indexed by the decreasing discretization dimension $m_o$. The four rows in each block show the values of the optimal error, the discretization error, and the respective truncation levels. In addition, the column corresponding to the a priori choice $m_o^{\rm apriori}$ mentioned above is shown in bold. 
We observe that for small noise (i.e., large SNR), the optimal error grows as $m_o$ decreases (i.e., as more and more measurements are averaged). However, for larger noise (smaller SNR), the optimal error remains nearly constant for larger $m_o$. Similarly, for larger noise, the optimal truncation index remains almost constant for larger $m_o$. This confirms that for larger noise, the reduced data gives the same accuracy. Moreover, we observe that if the discretization dimension $m_o$ is smaller than $m^{\rm apriori}_o$ (indicated in bold), the discretization error starts to dominate (i.e., the optimal error grows). The error obtained by the { discrepancy} principle is larger, in some cases substantially, than the optimal one. This is to be expected due to the rough upper bound on the data noise. Interestingly, we observe that the a priori determined discretization level is similar to the one where the discrepancy principle gives a maximum truncation level, which is mostly the level determined by Algorithm 1. These results hold regardless of the smoothness of the solution $f$. For general integral equations we use the open source software package ReguTools from \cite{Hansen:2007}. Here we choose the two test problems {\tt gravity} and {\tt heat}. We use the default parameters and {\tt example}$=2$ in {\tt gravity}.  These are discretizations of a gravity survey problem and the inverse heat equation using midpoint quadrature collocation.
For the dimension $m_o$ we get a system of linear equations $A_{m_o}x=b$ with $(A_{m_o})_{ij}=\frac{1}{m_o}\kappa(\xi_{im_o}, \xi_{jm_o}), i,j=1,...,m_o,$ and solution  $(x_{m_o})_i=f(\xi_{im_o}), i=1,...,m_o$. The grid points are $\xi_{im_o}=\frac{{-}1+2i}{2m_o}$, $i=1,. ..,m_o$. We choose the same initial numbers of function evaluations and discretization dimensions as for the previous test problem {\tt deriv2}. However, unlike the example above, we do not have explicit formulas for the singular value decompositions. Also, due to the fact that the integral kernels are more complicated, we can no longer evaluate $\|\tilde{f}_{k,m_{o}}^\delta\|$ exactly. Therefore, instead of constructing our approximation in $L^2(0,1)$ using the formula \eqref{t6e0}, we now solve the linear equations directly, i.e, we set $\overline{\tilde{f}}_{k,m_o}: =\sum_{j=1}^k\frac{(\overline{b}^\delta_{m_o},\tilde{w}_{jm_o})_{\R^{m_{o}}}}{\tilde{\sigma}_{jm_{o}}}\tilde{z}_{jm_o}\in\R^{m_o}$ (recall that $( \tilde{z}_{jm_o},\tilde{w}_{jm_o}, { \tilde{\sigma}_{jm_o}})$ is the singular value decomposition of $A_{m_o}$). { Hereby, $\overline{b}^\delta_{m_o}\in\R^{m_o}$ with $(\overline{b}^\delta_{m_o})_i =  \frac{1}{o} \sum_{t=1}^o ((b_m)_{(i-1)m_o+t} + \delta Z_{(i-1)m_o+t})$, $b_m=A_mx_m\in\R^m$ and $Z_1,...,Z_m$ are i.i.d. standard Gaussians.} The total error is then computed as follows: First, we have $e_{k,m_o}:=\|\overline{\tilde{f}}^\delta_{k,m_o}-x_{ m_o}\|_{\R^{m_o}}/\sqrt{m_o}$, which approximates the  error between our approximation and the projected solution.  The discretization error is then approximated by $e_{m_o}:=\|P_{m_o,D}x_{m_o} - x_D\|_{{\R}^D}/\sqrt{D}$, where $D=2m=8192$ and $P_{{m_o},D}x_{m_o}$ interpolates and extrapolates $x_{m_o}=\begin{pmatrix} f(\xi_{1m_o}) & ... & f(\xi_{m_om_o})\end{pmatrix}^T\in\R^{m_o}$ linearly to $\R^D$ as follows: { Let $\Lambda$ be the interpolating linear spline on the grid points $\xi_0=0,\xi_{1m_o},...,\xi_{m_om_o},\xi_1=1$ with $\Lambda(\xi_0)=\Lambda(\xi_1)=0$ and $\Lambda(\xi_{im_o}) = (x_{m_o})_i$ for $i=1,...m_o$. Then we set $(P_{m_o,D}x_{m_o})_i=\Lambda(\xi_{iD})$ for $i=1,...,D$.  Thus $e_{m_o}$ approximates the error between the projected solution and the exact one.} Consequently, the total error $\|\overline{\tilde{f}}^\delta_{k,m_o}-f\|_{L^2} = \sqrt{\|\overline{\tilde{f}}_{k,m_o}^\delta - f_m\|^2+\|f_m-f\|^2}$ is approximated by $\sqrt{e_{k,m_o}^2+e_{m_o}^2}$. The choice of linear interpolation is reasonable since the integral kernels are smooth. Note that we use ${\rm err}_{\rm sys}^2(m_o):=\frac{\|g'\|^2}{m_o}$ instead of the estimate ${\rm err}_{\rm sys}^2(m_o):=\|g''\|_\infty^2/(9\cdot 64 m_o^3)$. The reason is that in the implementation of the problems from  ReguTools, the right hand side is computed as an application of the quadrature discretization to the coefficient vector of the exact solution. The exact solution is continuous, but only piecewise smooth. Therefore, the latter estimate ${\rm err}_{\rm sys}^2(m_o):=\|g''\|_\infty^2/(9\cdot 64 m_o^3)$ falls below the modeling error and is replaced by the more conservative choice ${\rm err}_{\rm sys}^2(m_o):=\frac{\|g'\|^2}{m_o}$. The numerical results are shown in Table 3. They are very similar to the observations we made for {\tt deriv2}.

\begin{algorithm}[h!]\label{algorithm1}
	\caption{Discrepancy principle + dimension reduction}
	\begin{algorithmic}[1]
		\State Given noisy point evalutations $g^\delta_m\in\R^m$ { for $m=a^n$, with  $ a,n\in\N$ and $\N \ni n_0< n$};
		\State \textit{Initialization}
		\State $i=0$ and { $o_0=a^{n-n_0}$, $o_{i+1}= o_i/a$}
		\State Determine $k^\delta_{\rm dp}(m_{o_i})$ and $k^\delta_{\rm dp}(m_{o_{i+1}}{)}$
		\State \textit{Check stopping criterion}
		\While{ $k^\delta_{\rm dp}(m_{o_{i+1}})\ge k^\delta_{\rm dp}(m_{o_{i}})$ } 
        \If{{ $m_{o_{i+1}}= m$}}
        \State {$i=i+1$ and end loop}
        \EndIf
		\State $i=i+1$;
		\State $o_{i+1}=o_i{/ a}$;
		\State Determine $k^\delta_{\rm dp}(m_{o_{i+1}})$
		\EndWhile
		\State Final choice $\overline{f}^\delta_{k^\delta_{\rm dp},m_{o_i}}$
	\end{algorithmic}
\end{algorithm}
\begin{table}[h!]\label{tab1}
	\centering
	\caption{{\tt deriv2} with rough solution $(s=1/8)$, left columns, and medium smooth solution $(s=3/8)$, right columns}
	\setlength{\tabcolsep}{4pt}
	\begin{tabular}{cc|ccccc|ccccc|}
		&$m_o$ &$4^{6}$ & $4^{5}$ & $4^4$ & $4^3$ & $4^2$&$4^{6}$ & $4^{5}$ & $4^4$ & $4^3$ & $4^2$ \\
		\toprule
		${\rm SNR}=8^3$&	$e_{\rm opt}$ & {\bf 7.0e-1}	& 7.3e-1 & 8.0e-1 &  8.5e-1 &  8.9e-1 &{\bf 3.7e-2}	& 5.6-2 & 1.0e-1 &  1.7e-1 &  2.6e-1   \\
		&	$e_{\rm dp}$    & {\bf 8.0e-1}     & 8.3e-1 &8.7e-1 & 9.1e-1  &9.4e-1 & {\bf 9.9e-2}     & 1.4e-1 &1.8e-1 & 2.5e-1  &4.1e-1  \\
		&	$k_{\rm opt}$  & {\bf 65}   & 56& 23 & 10 &4 &{\bf 20}   & 14& 8 & 4 &2 \\
		&			$k_{\rm dp}$& {\bf 15}      & 9& 5 & 2 & 1 &  3      &{\bf 3}& 3 & 2 & 1 \\		\botrule		
		${\rm SNR}=8^2$&	$e_{\rm opt}$ & 7.6e-1	& {\bf 7.7e-1} & 7.9e-1 &  8.5e-1 &  8.9e-1 & 6.6e-2	&{\bf 7.0-2} & 1.0e-1 &  1.7e-1 &  2.6e-1  \\
		&	$e_{\rm dp}$    &  8.6e-1     &{\bf 8.5e-1} &8.7e-1 & 9.1e-1  &9.4e-1   &  1.8e-1     &{\bf 1.8e-1} &1.8e-1 & 2.5e-1  &4.1e-1\\
		&	$k_{\rm opt}$  & 29   & {\bf 29}& 22 & 10 &4 & 11   &{\bf 11}& 8 & 4 &2 \\
		&			$k_{\rm dp}$&  5      &{\bf 7}& 5 & 2 & 1&  3      &{\bf 3}& 3 & 2 & 1 \\		\botrule	
		${\rm SNR}=8$&	$e_{\rm opt}$ & 8.2e-1	& 8.2e-1 &{\bf 8.2e-1} &  8.5e-1 &  8.9e-1 & 1.2e-1	& 1.2-1 &{\bf 1.3e-1} &  1.8e-1 &  2.6e-1   \\
		&	$e_{\rm dp}$    &  9.1e-1     & 9.0e-1 &{\bf 8.9e-1} & 9.1e-1  &9.4e-1  & 4.1e-1     & 2.5e-1 &{\bf 2.5e-1} & 2.5e-1  &4.1e-1  \\
		&	$k_{\rm opt}$  & 13   & 13&{\bf 13} & 9 &4 & 6   & 6&{\bf 6} & 4 &2 \\
		&			$k_{\rm dp}$&  2      & 3&{\bf 3} & 2 & 1& 1      & 2&{\bf 2} & 2 & 1 \\		\botrule	
		${\rm SNR}=1$&	$e_{\rm opt}$ & 8.7e-1	& 8.8e-1 & 8.7e-1 & {\bf 8.8e-1} &  9.0e-1 & 2.1e-1	& 2.1-1 & 2.1e-1 & {\bf 2.1e-1} &  2.6e-1  \\
		&	$e_{\rm dp}$    &  1.0e0     & 9.5e-1 &9.4e-1 &{\bf 9.2e-1}  &9.4e-1  &  1e0     & 4.1e-1 &4.1e-1 &{\bf 4.1e-1}  &4.1e-1 \\
		&	$k_{\rm opt}$  & 5   & 5& 5 &{\bf 5} &4 & 3   & 3& 3 &{\bf 3} &2  \\
		&			$k_{\rm dp}$&  0      & 1& 1 &{\bf 2} & 1&  0      & 1& 1 &{\bf 1} & 1 \\		\botrule	
	\end{tabular}
\end{table}

\begin{table}[h!]\label{tab2}
	\centering
	\caption{{\tt deriv2} with smooth solution ($s=5/8$)}
	\setlength{\tabcolsep}{4pt}
	\begin{tabular}{cc|ccccc|}
		&$m_o$ &$4^{6}$ & $4^{5}$ & $4^4$ & $4^3$ & $4^2$ \\
		\toprule
		${\rm SNR}=8^3$&	$e_{\rm opt}$ &{\bf 6.0e-3}	& 1.5-2 & 3.4e-2 &  7.4e-2 &  8.1e-2   \\
		&	$e_{\rm dp}$    & {\bf 3.9e-2}     & 3.9e-2 &7.4e-2 & 1.9e-1  &1.9e-1  \\
		&	$k_{\rm opt}$  &{\bf 10}   & 6& 4 & 2 &2  \\
		&			$k_{\rm dp}$& {\bf 3}      & 3& 2 & 1 & 1 \\		\botrule	
		${\rm SNR}=8^2$&	$e_{\rm opt}$ & 1.4e-2	&{\bf 1.8-2} & 3.4e-2 &  7.4e-2 &  8.1e-2   \\
		&	$e_{\rm dp}$    &  7.4e-2     &{\bf 7.4e-2} &7.4e-2 & 1.9e-1  &1.9e-1  \\
		&	$k_{\rm opt}$  & 7   &{\bf 6}& 4 & 2 &2  \\
		&			$k_{\rm dp}$& 2      &{\bf 2}& 2 & 1 & 1 \\		\botrule	
		${\rm SNR}=8$&	$e_{\rm opt}$ & 3.3e-2	& 3.3-2 &{\bf 4.1e-2} &  7.5e-2 &  8.1e-2   \\
		&	$e_{\rm dp}$    &  1.9e-1     & 1.9e-1 &{\bf 7.7e-2} & 1.9e-1  &1.9e-1  \\
		&	$k_{\rm opt}$  & 4   & 4& {\bf4} & 2 &2  \\
		&			$k_{\rm dp}$&  1      & 1&{\bf 2} & 1 & 1 \\		\botrule	
		${\rm SNR}=1$&	$e_{\rm opt}$ & 8.4e-2	& 8.3-2 & 8.3e-2 & {\bf 8.7e-2} &  9.6e-2   \\
		&	$e_{\rm dp}$    & 1e0     & 1.9e-1 &1.9e-1 &{\bf 1.9e-1}  &1.9e-1  \\
		&	$k_{\rm opt}$  & 2   & 2& 2 &{\bf 2} &2  \\
		&			$k_{\rm dp}$&  0      & 1& 1 &{\bf 1} & 1 \\		\botrule		
	\end{tabular}
\end{table}
~\\
\begin{table}[h!]\label{tab4}
	\centering
	\caption{gravity, left columns, and heat, right columns}
	\setlength{\tabcolsep}{4pt}
	\begin{tabular}{cc|ccccc|ccccc|}
		&$m_o$ &$4^{6}$ & $4^{5}$ & $4^4$ & $4^3$ & $4^2$  &$4^{6}$ & $4^{5}$ & $4^4$ & $4^3$ & $4^2$  \\
		\toprule
		${\rm SNR}=8^3$&	$e_{\rm opt}$ &1.7e-2	&{\bf  1.7-2} & 2.0e-2 &  5.0e-2 &  1.9e-1 &{\bf 2.2e-2}	& 2.7-2 & 8.6e-2 &  3.4e-1 &  9.2e-1   \\
		&	$e_{\rm dp}$    &  6.1e-2     &{\bf  6.1e-2} &6.2e-2 & 1.6e-1  &2.8e-1  & {\bf 1.1e-1}     & 1.8e-1 &3.1e-1 & 6.5e-1  &1.0e0  \\
		&	$k_{\rm opt}$  & 11   &{\bf  11}& 11 & 11 &9 &{\bf 32}   & 32& 31 & 17 &3  \\
		&			$k_{\rm dp}$&  5      &{\bf  5}& 5 & 3 & 1& {\bf 16}      & 14& 9 & 6 & 2 \\		\botrule	
		${\rm SNR}=8^2$&	$e_{\rm opt}$ & 3.3e-2	& 3.3-2 &{\bf  3.5e-2} &  5.7e-2 &  1.9e-1 & 5.7e-2	&{\bf  5.9e-2} & 1.0e-1 &  3.4e-1 &  9.2e-1  \\
		&	$e_{\rm dp}$    &  1.2e-1     & 8.3e-2 &{\bf  1.2e-1} & 1.6e-1  &2.8e-1 &  3.0e-1     &{\bf  2.6e-1} &3.1e-1 & 6.5e-1  &1.0e0  \\
		&	$k_{\rm opt}$  & 9   & 9& {\bf  9} & 9 &8 & 24   & {\bf 24}& 24 & 17 &3  \\
		&			$k_{\rm dp}$& 4      &  5&{\bf  4} & 3 & 1 & 9      &{\bf  10}& 9 & 6 & 2\\		\botrule	
		${\rm SNR}=8$&	$e_{\rm opt}$ & 6.0e-2	& 6.0e-2 &  6.2e-2 &{\bf   7.7e-2} &  2.0e-1 &1.5e-1	& 1.5-1& {\bf 1.6e-1} &  3.5e-1 &  9.2e-1   \\
		&	$e_{\rm dp}$    &  2.1e-1     &2.1e-1 &1.5e-1 &{\bf  1.6e-1}  &2.8e-1  & 6.8e-1     & 6.0e-1 &{\bf 4.0e-1} & 6.5e-1  &1.0e0 \\
		&	$k_{\rm opt}$  & 6   &6& 6 & {\bf 6} &6   &16   & 16&{\bf 17} & 16 &3 \\
		&			$k_{\rm dp}$&  1      &1& 3 & {\bf 3} & 1& 3   &   6& {\bf 8} & 6 & 2  \\		\botrule	
		${\rm SNR}=1$&	$e_{\rm opt}$ & 1.3e-1	& 1.3-1 &1.3e-1 &  1.3e-1 &  {\bf 2.2e-1}  & 3.5e-1	&3.5e-1 & 3.6e-1 & {\bf  4.3-1} &  9.2e0   \\
		&	$e_{\rm dp}$    &  1e0     & 2.1e-1 &2.1e-1 & 2.2e-1 &{\bf  2.8e-1}  &  1.0e0     &8.9e-1 & 7.5e-1 & {\bf 6.8e-1}  &1.0e0 \\
		&	$k_{\rm opt}$  & 4   & 4&4 &  4&{\bf 4} & 10   &10& 10 &{\bf  10} &3 \\
		&			$k_{\rm dp}$&  0      & 1& 1 & 1 &{\bf  1} &  0      &1& 2 &{\bf  3} & 2\\		\botrule		
	\end{tabular}
\end{table}
\newpage
\
\section{Concluding Remarks}\label{s7}

In this paper, we introduced and analyzed a novel approach for solving ill-posed integral equations, with a focus on reducing the necessary computational cost. We obtained rigorous error bounds, and designed and implemented an adaptive method that performed promisingly and stably. Note that also in further experiments, where we tested the setup with an asymmetric heavy-tailed distribution instead of a Gaussian one, we observed no significant differences and obtained similar results. 

We point out three important issues for further research: First, the important task of constructing adaptive data-driven methods has only been touched upon briefly. In particular, it would be advantageous to have a method that does not require knowledge of the noise level $\delta$ and the norm of $g'$. A possible candidate would be the (modified) heuristic discrepancy principle as introduced in \cite{jahn2023noise}, for which good results have been demonstrated. We have already tested this method numerically. While for fixed $m_o$ the achieved error was smaller than that of the discrepancy principle, it was not possible to identify the optimal discretization dimension as the maximum. A possible reason could be that the simple heuristic discrepancy principle is not suitable for the systematic error resulting from the averaging of the true data, and further modification might be needed.  Another promising approach would be to integrate the averaging into the forward operator itself, as was already done in \cite{jahn2022discretisation}. This has the advantage that the data error of the averaged data no longer has a systematic error component, since we now have to consider the error $\overline{g}^\delta_{m_o} - \overline{g}_{m_o}$.  Consequently, knowledge of $\|g'\|$ or $\|g''\|$ is no longer required. However, a rigorous analysis is more complicated, since the effect of the averaging step on the approximation error must be treated carefully, i.e. the effect of the averaging of the operator on the (discretized) singular value decomposition must be well controlled.

In this article, we deliberately focused on the one-dimensional situation only. This allowed us to explain our new averaging method without too much technical detail, to derive precise error estimates, and to demonstrate its superior numerical properties in a straightforward way. Of course, the next step is to extend our method to the more practically relevant higher-dimensional problems that arise, for example, in imaging science. {  Moreover}, note that instead of spectral cut-off, other regularization methods could be used, such as standard ones (Tikhonov regularization or Landweber iteration) and in particular more advanced techniques such as conjugate gradient or stochastic gradient descent. {  While the analysis is straightforward for the specific kernel \eqref{s1:e1a}, it is much more difficult for the general situation. The main problem is that semi-discrete singular functions belonging to high-frequency components do not approximate the corresponding continuous ones. Here one has to consider the analysis with bounds based on the explicit representations of the estimator.  Another limitation  is that we considered a fixed design for the measurement points. In important scenarios, e.g. in the context of inverse learning \cite{rastogi2023inverse}, it may be more appropriate to assume that the locations of the measurement points are sampled from a distribution. Then the projection of the measurements onto the collocation grid is more challenging. Here, simple binning of the points may not be the best approach, and weighted averages should be used instead. In this case, a trade-off must be made between the better approximation of the systematic component and the increased variance (compared to the mean) of the weighted mean. Finally, it would be interesting to investigate higher order quadrature rules depending on the smoothness of the kernel.}

\appendix

\section{Appendix}

In this section, we collect the proofs of our theoretical results.
\subsection{Proofs of Section \ref{s3:s1}.}
We begin with the proofs for Section \ref{s3:s1}.
\begin{proof}[Proof of Theorem \ref{t2}]
Note that $\mathcal{N}(K_m)^\perp$ is spanned by linear splines on the grid defined by $\xi_{1m},...,\xi_{mm}$. As mentioned above, $s>3/4$ implies that $f$ is differentiable and we get the bound
\begin{equation}\label{s2:err1}
\Delta_m=\|f_m-f\|\le \frac{\|f'\|}{\sqrt{2}(m+{1)}},
\end{equation}
see (3.12) in \cite{jahn2022cross}.   First, it suffices to restrict here to the case when $o_m>1$, since $\overline{f}_{k,m_{o}}^\delta$ and $f_{k,m}^\delta$ coincide for $o=1$. By definition this further implies ${(m+1)} \delta^2\ge \rho^2$. Thus
	\begin{equation}\label{s2:err1a}
		\frac{\|f'\|^2}{(m{+1})^2} \le \frac{\rho^2}{(m{+1)}^2}\le \frac{\delta^2}{m+1}\le \frac{k^5\delta^2}{m{+1}}
	\end{equation}
		for all $k=1,...,m$, and 
  \begin{equation*}
      \frac{\|f'\|^2}{(m+1)^2}\le \frac{\rho^2}{(0+1)^{4s}}
  \end{equation*}
  for $k=0$. Thus we see that the discretization error $\Delta_m$ is dominated by the other error contributions if $o_m\ge 2$.  We now  show that
		\begin{align}\label{t2e1}
		c' \left(\frac{k^5}{m{+1}}\delta^2 + (k{+1)}^{-{ 4}s}\rho^2\right)\le \sup_{f\in\mathcal{X}_{s,\rho}}\E\|f^\delta_{k,m}-f\|^2\le C'\left(\frac{k^5}{m{+1}}\delta^2 + {(k+1)}^{-{4}s}\rho^2\right)
	\end{align}
		with
	\begin{align*}
		c':&= \frac{16}{\pi^{4(s+1)}} \qquad C':= 3\pi^4 + \frac{1}{2}.
	\end{align*}
	For the variance term $\delta^2\sum_{j=1}^k\frac{1}{\sigma_{jm}^2}$,  Lemma \ref{lem0} and the elementary estimate $2x/\pi\le \sin(x)\le x$ for $x\le \pi/2$ gives
	\begin{equation}\label{s2:err2}
	\frac{2^4k^5}{{5(m+1)}} \le \sum_{j=1}^k\frac{1}{\sigma_{jm}^2} \le \frac{3\pi^4 k^5}{ {(m+1)}},
	\end{equation}
  where we also used that $\frac{k^5}{5}\le\sum_{j=1}^kj^{4}\le k^5.$
	For the approximation error $\sum_{j=k+1}^m(f,v_{jm})^2$, we have by (3.6) in \cite{jahn2022cross} that
	\begin{equation}\label{s2:err3}
	\sup_{f\in\mathcal{X}_{s,\rho}}\sum_{j=k+1}^m(f,v_{jm})^2\le \frac{3^{s+1}}{2^{4s-1}}\frac{\rho^2}{{(k+1)}^{4s}}\le 3\pi^4 \frac{\rho^2}{(k+1)^{4s}},
	\end{equation}
	which together with \eqref{s2:e11}, \eqref{s2:err1}, \eqref{s2:err1a} and \eqref{s2:err2} gives the upper bound in \eqref{t2e1}. For the lower bound we use the specific instance $\tilde{f}:=\varphi_{ s}(\sigma_{k+1}^2)\rho v_{k+1}$. Since $f-f_m$ is orthogonal to the range of $v_{1m},...,v_{mm}$, we have
		\begin{align}\label{s3e00}
		\sup_{f\in\mathcal{X}_{s,\rho}}\E\|f_{k,m}^\delta-f\|^2&\ge \delta^2\sum_{j=1}^k\frac{1}{\sigma_{jm}^2} +\sum_{j=k+1}^m(\tilde{f},v_{jm})^2.
	\end{align}
		All that remains is to bound the second term from below. Due to the special choice of $\tilde{f}$ and {Proposition} 3.8 from \cite{jahn2022cross} we get, for $k<m$,
		\begin{align}\label{s3:e01}
\sum_{j=k+1}^m(\tilde{f},v_{jm})^2&=\varphi_s(\sigma_{k+1}^2)^2\rho^2 \frac{(m+1)\sigma_{k+1}^2}{\sigma_{k+1m}^2}\ge\frac{{16}\rho^2}{\pi^{4(s+1)}{(k+1)}^{4s}},
	\end{align}
since $s>3/4$. In case of $k=m$ we have $\sum_{j=k+1}^m(\tilde{f},v_{jm})^2=0$ and the preceding estimate does not hold. However, since $(m+1)\delta^2\ge \rho^2$ and $s>3/4$, we obtain for $k=m$ that
 \begin{align*}
     \sup_{f\in\mathcal{X}_{s,\rho}}\E\|f^\delta_{m,m} - f\|^2 \ge \frac{16}{5}\frac{m^5}{m+1}\delta^2 \ge \frac{8}{5} \frac{m^5}{m+1} \delta^2  + \frac{8}{5} \frac{m^5\rho^2}{(m+1)^2} \ge c'\left( \frac{m^5}{m+1}\delta^2 +  \frac{\rho^2}{(m+1)^{4s}}\right)
 \end{align*}
 which shows the assertion \eqref{t2e1}.
	
	Next, let $2\le o \in O_{m,\delta,\rho}$. Clearly, the components of $\overline{g}^\delta_{m_o}$ are approximations of $Kf=g$ on the coarser grid with meshwidth $1/(m_o+1)$ instead of \mbox{$1/(m+1)$}. More precisely,
\begin{equation*}
		\left(\overline{g}^\delta_{m_o}\right)_i = \frac{1}{o}\sum_{j=1}^o (g\left(\xi_{{((i-1)o+j )}m}\right) + \delta Z_{(i-1)o+j})=\frac{1}{o}\sum_{j=1}^o g\left(\xi_{((i-1)o+j)m}\right) + \frac{\delta}{\sqrt{o}}\overline{Z}_i,
		\end{equation*}
		where we defined $\overline{Z}_i:= \frac{\sum_{j=1}^o Z_{(i-1)o+j}}{\sqrt{o}}$. Note that $\overline{Z}_i\sqrt{o}$ and $Z_1$ have the same { expectation and variance}. Furthermore, the $\overline{Z}_i$ are independent and identically distributed. Consequently, the variance of the measurement is reduced by a factor of $1/o$. However, the bias has changed. Now, using the Cauchy-Schwarz inequality, we obtain the upper bound
		\begin{align}\notag
		\left|\frac{\sum_{j=1}^o g(\xi_{{((i-1)o + j)}m})}{o} - g(\xi_{im_o})\right|
		&\le \frac{1}{o}\left|\sum_{j=1}^og(\xi_{{((i-1)o+j)}m})-g(\xi_{im_o})\right|\le \sup_{t\in(\xi_{(i-1)o,m},\xi_{i,m})}\left|g(t)-g(\xi_{im_o})\right|\\\label{s3e01}
		&\le \left|\int_{\xi_{{((i-1)o)}m}}^{\xi_{{(io)}m}}g'(t){\rm d}t\right|
		\le \sqrt{\frac{{o}}{{(m+1)}}} \sqrt{\int_{\xi_{{((i-1)o)}m}}^{\xi_{{(io)}m}}(g')^2(t){\rm d}t}.
	\end{align}
		Consequently, for $k=0,...,m_{o}$, and using \eqref{s2:err2}, we obtain
		\begin{align*}
		&\sum_{j=1}^k\frac{\E\left( \overline{z}^\delta_{m_o} - K_{m_o}f,u_{jm_o}\right)^2}{\sigma_{jm_o}^2}\\
		  &~~= \sum_{j=1}^k\frac{1}{\sigma_{jm}^2}\left( \begin{pmatrix} \frac{\sum_{s=1}^o g\left(\xi_{sm}\right)}{o}-g(\xi_{1m_o}) \\ \vdots \\ \frac{\sum_{s=1}^og(\xi_{{((m-1)o+s)}m})}{o} -  g(\xi_{m_om_o}) \end{pmatrix}, u_{jm_o}\right)^2 + \sum_{j=1}^k\frac{\delta^2}{\sigma_{jm_o}^2}\E\left[\left(\begin{pmatrix} \bar{Z}_1 \\ \vdots\\ \bar{Z}_{m_o} \end{pmatrix},u_{jm_o}\right)^2\right]\\
		&~~\le \sigma_{km_o}^{-2} \left\| \begin{pmatrix} \frac{\sum_{s=1}^o g\left(\xi_{sm}\right)}{o}-g(\xi_{1m_o}) \\ \vdots \\ \frac{\sum_{s=1}^og(\xi_{{((m-1)o+s)}m})}{o} -  g(\xi_{m_om_o}) \end{pmatrix} \right\|^2+ \frac{\delta^2}{o}\sum_{j=1}^{k}\frac{1}{\sigma_{jm_o}^2}\\
  &~~\le \sigma_{km_o}^{-2}\frac{o}{{m+1}} \sum_{i=1}^{m_o} \int_{\xi_{{((i-1)o+1)}m}}^{\xi_{{(io)}m}} |g'(x)|^2 {\rm d}x + \frac{3\pi^4 k^5\delta^2}{o(m_o+1)}\le \frac{3\pi^4 k^4}{m_o+1}\left(\frac{o}{m+1}\|g'\|^2 +\frac{ k\delta^2}{o}\right).
	\end{align*}
	 Using $\sup_{f\in\mathcal{X}_{s,\rho}}\|g'\|\le \rho$ together with \eqref{s2:err1} and \eqref{s2:err3} (where $m$ is replaced by $m_o$),   the error gets
		\begin{align*}
		&\sup_{f\in\mathcal{X}_{s,\rho}}\E\|\overline{f}^\delta_{k,m_o} - f\|^2\\
		&~~~\le ~\sup_{f\in\mathcal{X}_{s,\rho}}\sum_{j=1}^k \frac{\E\left(\overline{z}^\delta_{m_o}-K_{m_o}f,u_{jm_o}\right)^2}{\sigma_{jm_o}^2} + \sup_{f\in\mathcal{X}_{s,\rho}}\sum_{j=k+1}^{{m_o}} (f,v_{jm_o})^2 + \sup_{f\in\mathcal{X}_{s,\rho}}\|f_{m_o}-f\|^2\\
		&~~~\le  ~\frac{3\pi^4 k^4}{m_o+1}\left(\sup_{f\in\mathcal{X}_{s,\rho}}\frac{o}{m+1}\|g'\|^2+k\frac{\delta^2}{o}\right) + \frac{3^{s+1}}{2^{4s-1}} \frac{\rho^2}{(k+1)^{4s}} + \frac{\rho^2}{2(m_o+1)^2}\\
		&~~~\le  ~3\pi^4 k^4\left(\frac{o^2}{(m+1)^2}\rho^2+k\frac{\delta^2}{m+1}\right) + \frac{3^{s+1}}{2^{4s-1}}\frac{\rho^2}{(k+1)^{4s}} + \frac{\rho^2}{2(m_o+1)^2}\\
  &~~~\le 3\pi^4k^4\left(\frac{3o^2}{2(m+1)^2}\rho^2 + k \frac{\delta^2}{m+1}\right)+\frac{3^{s+1}}{2^{4s-1}} \frac{\rho^2}{(k+1)^{4s}}.
	\end{align*}
		  For $o_m\ge o\ge 2$, it holds that $\frac{o^2}{(m+1)^2}\rho^2 \le \frac{\delta^2}{m+1}$. Consequently,
		\begin{align}\label{upbound}
		\sup_{f\in\mathcal{X}_{s,\rho}}\E\|\overline{f}^\delta_{k,m_o}-f\|^2 &\le { \frac{15}{2}}\pi^4 k^5 \frac{\delta^2}{m+1} + \frac{3^{s+1}}{2^{4s-1}} \frac{\rho^2}{(k+1)^{4s}} \le \frac{15}{2} \pi^4 \left(k^5 \frac{\delta^2}{m+1} + \frac{\rho^2}{(k+1)^{4s}}\right)
  \end{align}
For a lower bound we first restrict to $k<m_o$ and  obtain, with $\overline{g}_{m_o} = \overline{g}_{m_o}^{0}$ (averages of exact $g_m$), 
		\begin{align*}
	&\sup_{f\in\mathcal{X}_{s,\rho}}\E\|\overline{f}^\delta_{k,m_o} - f\|^2\\
		&~~~=~\frac{\delta^2}{o}\sum_{j=1}^k\frac{1}{\sigma_{jm_o}^2} + \sup_{f\in\mathcal{X}_{s,\rho}}\left(\sum_{j=1}^k\frac{(\overline{g}_{m_o}-g_{m_o},u_{jm_o})^2}{\sigma_{jm_o}^2} + \sum_{j=k+1}^{m_o}(f,v_{jm_o})^2 +\|f-f_{m_o}\|^2\right)\\
		&~~~\ge ~\frac{\delta^2}{o}\sum_{j=1}^k \frac{1}{\sigma_{jm_o}^2} + \sup_{f\in\mathcal{X}_{s,\rho}} \sum_{j=k+1}^{m_o}(f,v_{jm_o})^2
		\ge \frac{16}{5} \delta^2\frac{k^5}{o(m_o+1)} + \rho^2 \frac{16}{\pi^{4(s+1)}(k+1)^{4s}}\\
  	&~~~\ge \frac{8}{5} \delta^2\frac{k^5}{(m+1)} + \rho^2 \frac{16}{\pi^{4(s+1)}(k+1)^{4s}}\\
		&~~~\ge ~c'\left(k^5 \frac{\delta^2}{m+1} + \frac{\rho^2} {(k+1)^{4s}}\right)
	\end{align*}
		where we used $o\le m$ in the fourth step. For $k=m_o$, since $\sum_{j=m_o+1}^{m_o} (f,v_{jm_o})^2=0$, we first have 
\begin{align*}
  \sup_{f\in\mathcal{X}_{s,\rho}}\E\|\overline{f}^\delta_{k,m_o} - f\|^2\ge \frac{16}{5} \delta^2 \frac{m_o^5}{o(m_o+1)} = \frac{8}{5}\left(\delta^2 \frac{m_o^5}{o(m_o+1)} + \delta^2 \frac{m_o^5}{o(m_o+1)}\right).
  \end{align*}
  Because $o^2\le (m+1)\delta^2/\rho^2$ and $s>3/4$, it holds that
  \begin{align*}
      \frac{m_o^5}{o(m_o+1)}\delta^2 \ge \frac{m_o^5 o\rho^2}{(m_o+1)(m+1)}\ge m_o^5 \frac{\rho^2}{(m_o+1)^2} \ge \rho^2(m_o+1)^{-4s},
  \end{align*}
  and thus, for all $k=0,...,m_o$,
  \begin{equation}\label{lowbound}
      \sup_{f\in\mathcal{X}_{s,\rho}}\E\|\overline{f}^\delta_{k,m_o}-f\|^2 \ge c'\left(k^5 \frac{\delta^2}{m+1} + \rho^2 (k+1)^{-4s}\right).
  \end{equation}
  Finally, we have
  \begin{align*}
\min_{k=0,...,m}\sup_{f\in\mathcal{X}_{s,\rho}}\E\|f_{k,m}^\delta-f\|^2&\le C' \min_{k=0,...,m}\left(k^5 \frac{\delta^2}{m+1} + \frac{\rho^2}{(k+1)^{4s}}\right)\le C'\min_{k=0,...,m_o}\left(k^5 \frac{\delta^2}{m+1} + \frac{\rho^2}{(k+1)^{4s}}\right)\\
& \le\frac{C'}{c'}\min_{k=0,...,m_o}\sup_{f\in\mathcal{X}_{s,\rho}}\E\|\overline{f}^\delta_{k,m_o}-f\|^2,
  \end{align*}
  which implies that $c{\rm err}(\delta,m,s,\rho)\le \overline{{\rm err}}(\delta,m_o,s,\rho)$, and
  where we have used \eqref{t2e1} in the first and \eqref{lowbound} in the third step. Similarly,
   \begin{align*}
\min_{k=0,...,m}\sup_{f\in\mathcal{X}_{s,\rho}}\E\|f_{k,m}^\delta-f\|^2&\ge c' \min_{k=0,...,m}\left(k^5 \frac{\delta^2}{m+1} + \frac{\rho^2}{(k+1)^{4s}}\right)\ge \frac{c'}{2}\min_{k=0,...,m_o}\left(k^5 \frac{\delta^2}{m+1} + \frac{\rho^2}{(k+1)^{4s}}\right)\\
& \ge\frac{c'}{15\pi^4}\min_{k=0,...,m_o}\sup_{f\in\mathcal{X}_{s,\rho}}\E\|\overline{f}^\delta_{k,m_o}-f\|^2,
  \end{align*}
  implies $\overline{{\rm err}}(\delta,m_o,s,\rho) \le C {\rm err}(m,\delta,s\rho)$. Here, in the second step we employed the fact that 
  \begin{equation*}
  {k'}^5 \frac{\delta^2}{m+1} + \frac{\rho^2}{(k'+1)^{4s}} \ge m_o^5 \frac{\delta^2}{m+1} \ge \frac{1}{2}\left( m_o^5 \frac{\delta^2}{m+1} + \frac{\rho^2}{(m_o+1)^{4s}}\right)
\end{equation*}
 for all $m_o<k'\le m$, since $o^2 \rho^2\le (m+1)\delta^2$ and $s>3/4$. Theorem \ref{t2} is finally proved.\hfill
\end{proof}
\subsection{Proofs of Section \ref{s5}}
To prove the main theorems, we need some auxiliary results. The corresponding lemmata were given above, and before we give their proofs we state and prove a slight generalization of Proposition \ref{prop0}. Since the kernel $\kappa$ is non-symmetric in general, we need the following semi-discrete model of $K^*$, defined similarly to $K_m$ in \eqref{s1:e2:a}, but where the integration is with respect to the first argument:
\begin{align}\notag
	Q_m:L^2(0,1)&\to\R^m\\\label{disc:adj}
	g &\mapsto \frac{1}{\sqrt{m}} \left( (K^*g)(\xi_{lm})\right)_{l=1}^m = \frac{1}{\sqrt{m}}\left(\int \kappa(x,\xi_{lm}) g(x)\mathrm{d}x\right)_{l=1}^m.
\end{align}
\begin{proposition}\label{prop0mod}
    Assume that $\kappa$ fulfills Assumption \ref{a1}. Let $(\sigma_{jm}')_{j=1}^m$ be the singular values of the semi-discrete operator $Q_m$ from \eqref{disc:adj} and let $\rho_{jm}$ be the eigenvalues of the normalized design matrix $S_m$ given in \eqref{designmat}. Then it holds that $\sigma_{jm}' = \sqrt{\rho_{jm}}$ for all $j=1,...,m$. 
\end{proposition}
\begin{proof}[Proof of Proposition \ref{prop0mod}.]
This follows directly from the fact that $S_m \alpha = Q_m Q_m^* \alpha$ for all $\alpha \in \R^m$.
\hfill
\end{proof}
We now start to prove the lemmata.
\begin{proof}[Proof of Lemma \ref{lem00}.]
	The proof is based on the Courant-Fischer principle, which states that for {  self-adjoint compact positive semi-definite} operators $S,T$ in Hilbert spaces there holds $|\lambda_i(S)-\lambda_i(T)|\le \|S-T\|$, where $\lambda_i(\cdot)$ is the $i$-th largest eigenvalue. We first apply this bound to $KK^*$ and $Q_m^*Q_m$, where $Q_m$ is from \eqref{disc:adj}. {   It is well-known from \cite{nair2007regularized} that
 \begin{align}\label{lem00e1b}
        \|KK^* - Q_m^*Q_m\| &=\sup_{\substack{h\in L^2(0,1)\\\|h\|=1}}\|KK^*h - Q_m^* Q_m h\| \le \frac{C_K^2}{6m^2}.
	\end{align} 
For the readers convenience we give a short proof of this assertion. Let $h\in L^2(0,1)$, then we have
\begin{align*}
    KK^* h - Q_m^*Q_m h&= \int h(z)\left( \int \kappa(\cdot,y)\kappa(z,y) {\rm d}y - \frac{1}{m}\sum_{l=1}^m \kappa(\cdot,\xi_{lm})\kappa(z,\xi_{lm})\right){\rm d}z
\end{align*}
and thus taking the $L^2$ norm and using the Cauchy-Schwarz inequality  yields
\begin{align*}
    \left\|KK^*h-Q_m^*Q_mh\right\|^2&=\int\left(\int h(z)\left( \int \kappa(x,y)\kappa(z,y) {\rm d}y - \frac{1}{m}\sum_{l=1}^m \kappa(x,\xi_{lm})\kappa(z,\xi_{lm})\right){\rm d}z\right)^2 {\rm d}x\\
    &\le \|h\|^2 \int\left(\int \kappa(x,y)\kappa(z,y){\rm d}y-\frac{1}{m}\sum_{l=1}^m\kappa(x,\xi_{lm})\kappa(z,\xi_{lm})\right)^2{\rm d}z{\rm d}x\\
    &\le \frac{\|h\|^2}{24^2m^4} \sup_{x,y}\left\|\partial^2_{y^2}\left(\kappa(x,y)\kappa(z,y)\right)\right\|_\infty^{{ 2}}\le \frac{\|h\|^2 4^2 C_K^{{ 4}}}{24^2 m^4}, 
\end{align*}
and the assertion \eqref{lem00e2} is proved. }  For $w\in\R^m$ we have
	\begin{align*}
		&\left( S_m w - A_m^T A_m w\right)_i\\
		 &~~~= ~ \frac{1}{m}\sum_{j=1}^m \int \kappa(x,\xi_{im})\kappa(x,\xi_{jm}){\rm d}x~ w_j - \frac{1}{m^2} \sum_{j=1}^m \sum_{l=1}^m \kappa(\xi_{lm},\xi_{im})\kappa(\xi_{lm},\xi_{jm})w_j\\
		&~~~=~\frac{1}{m}\sum_{j=1}^m w_j \left( \int \kappa(x,\xi_{im})\kappa(x,\xi_{jm}) {\rm d}x - \frac{1}{m}\sum_{l=1}^m \kappa(\xi_{lm},\xi_{im})\kappa(\xi_{lm},\xi_{jm})\right)
	\end{align*}
	and thus
	\begin{align}\notag
		&\|S_m - A_m^TA_m\|
		 =\sup_{\substack{w\in\R^m, \|w\|=1}} \|S_m w- A_m^TA_m w\|\\\notag
		&~~~=~\sup_{\substack{w\in\R^m, \|w\|=1}}\sqrt{\sum_{i=1}^m\left(\frac{1}{m}\sum_{j=1}^m w_j \left(\int \kappa(x,\xi_{im})\kappa(x,\xi_{jm}){\rm d}x - \frac{1}{m}\sum_{l=1}^m\kappa(\xi_{lm},\xi_{im})\kappa(\xi_{lm},\xi_{jm})\right)\right)^2}\\\notag
		&~~~\le~ \sup_{\substack{w\in\R^m, \|w\|=1}} \sqrt{\sum_{i=1}^m\left(\frac{1}{m}\sum_{j=1}^m|w_j| \frac{\|\partial_x^2\left(\kappa(x,\xi_{im})\kappa(x,\xi_{jm})\right)\|_\infty}{24 m^2}\right)^2}\\\label{lem00e2}
  &~~~\le\sup_{\substack{w\in\R^m, \|w\|=1}} \frac{C_K^2}{6 m^2} \frac{1}{\sqrt{m}}\sum_{j=1}^m |w_j|
		\le \frac{C_K^2}{6 m^2}.
	\end{align}

	By Proposition \ref{prop0mod}, we know that the singular values of the semi-discrete operator $Q_m$ are the square root of the eigenvalues of the matrix $S_m$, denoted by $\rho_{jm}$. Therefore, applying the Courant-Fischer principle and \eqref{lem00e2}, we deduce that 
	\begin{equation*}
		\left|\rho_{jm}-\tilde{\sigma}_{jm}^2\right|\le \|S_m - A_m^TA_m\|\le \frac{C_K^2}{6 m^2}\quad\mbox{for all } j=1,...,m.
	\end{equation*}
	Furthermore, by \eqref{lem00e1b} we also infer
	\begin{equation*}
		\left|\sigma_j^2-\rho_{jm}\right|\le \|KK^*-Q_m^*Q_m\|\le \frac{C_K^2}{6 m^2}\quad\mbox{for all } j=1,...,m.
	\end{equation*}
	 With the triangle inequality we get
	\begin{equation*}
		\left|\sigma_j^2-\tilde{\sigma}_{jm}^2\right|\le \frac{C_K^2}{3 m^2},
	\end{equation*}
	and the proof of Lemma \ref{lem00} is finished.\hfill
	\end{proof}
 \begin{proof}[Proof of Lemma \ref{lem1}.]
	The proof of \eqref{lem1:e01} is straightforward: We have
	\begin{align*}
		\varepsilon^2&\ge \|Kv-\lambda v\|^2 = \| \sum_{i\in\N} \lambda_i(v,v_i)v_i -\lambda\sum_{i\in\N} (v,v_i)v_i\|^2 = \sum_{i\in {\N}}(\lambda_i-\lambda)^2(v,v_i)^2\\ &\ge \min_{i\in\N}|\lambda_i-\lambda|^2 \sum_{i\in\N}(v,v_i)^2
		=\|v\|^2\min_{i\in\N}|\lambda_i-\lambda|^2.
	\end{align*}
	For \eqref{lem1:e02}, {  it holds that
	\begin{align*}
		(v,Pv) &= \left(v,\sum_{i\in I}(v,v_i)v_i\right)= \sum_{i\in I}(v,v_i)^2=\|v\|^2-\sum_{i\not\in I}(v,v_i)^2 =\|v\|^2- \sum_{i\not\in I}\frac{(\lambda_i-\lambda)^2}{(\lambda_i-\lambda)^2}(v,v_i)^2\\
  &\ge \|v\|^2- \frac{1}{c^2}\sum_{i\not\in I}(\lambda_i-\lambda)^2(v,v_i)^2\\
		&\ge \|v\|^2- \frac{1}{c^2}\sum_{i\in\N}(\lambda_i-\lambda)^2(v,v_i)^2 = \|v\|^2 - \frac{1}{c^2}\|Kv-\lambda v\|^2 \ge \|v\|^2-\frac{\varepsilon^2}{c^2}.
	\end{align*}}\hfill
	\end{proof}
 \begin{proof}[Proof of Lemma \ref{lem002}.]
	We start with the proof of \eqref{lem002:e1}. First,
	\begin{align*}
		&(\tilde{v}_{jm},\tilde{v}_{im})= \frac{1}{\tilde{\sigma}_{jm}\tilde{\sigma}_{im} m } \sum_{l,l'=1}^m (\tilde{w}_{jm})_l (\tilde{w}_{im})_{l'} \int \kappa(\xi_{lm},y)\kappa(\xi_{l'm},y){\rm d}y\\
		&~~~=~\frac{1}{\tilde{\sigma}_{jm}\tilde{\sigma}_{im} m} \sum_{l,l'=1}^m(\tilde{w}_{jm})_l(\tilde{w}_{im})_{l'}\left(\frac{1}{m}\sum_{t=1}^m \kappa(\xi_{lm},\xi_{tm})\kappa(\xi_{l'm},\xi_{tm})  + \mathcal{O}\left(\frac{1}{m^2}\right)\right)\\
		&~~~=~\frac{1}{\tilde{\sigma}_{jm}\tilde{\sigma}_{im}}\sum_{t=1}^m \left(\frac{1}{m}\sum_{l=1}^m \kappa(\xi_{lm},\xi_{tm})(\tilde{w}_{jm})_l\right)\left(\frac{1}{m}\sum_{l'=1}^m\kappa(\xi_{l'm},\xi_{tm}) (\tilde{w}_{im})_{l'}\right)\\
		&\qquad + \frac{1}{\tilde{\sigma}_{jm}\tilde{\sigma}_{im}} \sum_{l=1}^m \frac{|(\tilde{w}_{jm})_l|}{\sqrt{m}}\sum_{l'=1}^m\frac{|(\tilde{w}_{im})_{l'}|}{\sqrt{m}}\mathcal{O}\left(\frac{1}{m^2}\right)\\
		&~~~=~\frac{1}{\tilde{\sigma}_{jm}\tilde{\sigma}_{im}} \sum_{t=1}^m \tilde{\sigma}_{jm} (\tilde{z}_{jm})_t \tilde{\sigma}_{im} (\tilde{z}_{im})_t + \frac{1}{\tilde{\sigma}_{jm}\tilde{\sigma}_{im}} \mathcal{O}\left(\frac{1}{m^2}\right) = \delta_{ij}+\frac{1}{\tilde{\sigma}_{jm}\tilde{\sigma}_{im}}\mathcal{O}\left(\frac{1}{m^2}\right).
	\end{align*}
	Since $\|\partial^2_y( \kappa(\xi_{lm},y)\kappa(\xi_{l'm},y))\|_\infty\le 4C_K^2$ and $\sum_{l=1}^m \frac{|(\tilde{w}_{jm})_l|}{\sqrt{m}}\le 1$, we can write the constant exactly in the $\mathcal{O}$ notation {  as $\frac{4 C_K^2}{24 m^{2}}$. Together with the fact that $\tilde{\sigma}_{lm}^2\ge\sigma_l^2/2$ for $l\le J_m$, we obtain
	\begin{equation*}
		\left|(\tilde{v}_{jm},\tilde{v}_{im}) - \delta_{ij}\right|\le \frac{1}{\tilde{\sigma}_{jm}\tilde{\sigma}_{im}}\frac{4 C_K^2}{24 m^2}\le \frac{C_K^2}{3 \sigma_j \sigma_i m^2}	
	\end{equation*} 
	and \eqref{lem002:e1}. This brings us to the proof of \eqref{lem002:e2}. The upper bound is clear from the definition of $P\tilde{v}_{jm}$.} For the lower bound we use {Lemma} \ref{lem1}. As a first step, we show that
	\begin{equation}\label{lem1:e1}
		\left\|K^*K \tilde{v}_{jm} - \tilde{\sigma}_{jm}^{  2} \tilde{v}_{jm}\right\| \le \frac{C_K^3}{2 \sigma_j m^2}.
	\end{equation}
	Indeed,
	\begin{align}\notag
		&K^*K\tilde{v}_{jm}=\int {\kappa}(y,\cdot) \int \kappa(y,z) \tilde{v}_{jm}(z) {\rm d}z{\rm d}y\\\notag
		&~= ~\int \kappa(y,\cdot)\left[\int \kappa(y,z) \tilde{v}_{jm}(z){\rm d}z - \frac{1}{m}\sum_{l=1}^m \kappa(y,\xi_{lm}) \tilde{v}_{jm}(\xi_{lm})\right]{\rm d}y + \int \kappa(y,\cdot) \frac{1}{m}\sum_{l=1}^m \kappa(y,\xi_{lm}) \tilde{v}_{jm}(\xi_{lm}){\rm d}y\\\label{lem1:e1a}
		&~= ~\int \kappa(y,\cdot)\left[\int \kappa(y,z) \tilde{v}_{jm}(z){\rm d}z - \frac{1}{m}\sum_{l=1}^m \kappa(y,\xi_{lm}) \tilde{v}_{jm}(\xi_{lm})\right]{\rm d}y\\\label{lem1:e1b}
		&\qquad +\frac{1}{m}\sum_{l=1}^m \tilde{v}_{jm}(\xi_{lm})\left[\int \kappa(y,\cdot) \kappa(y,\xi_{lm}){\rm d}y - \frac{1}{m}\sum_{l'=1}^m  \kappa(\xi_{l'm},\cdot) \kappa(\xi_{l'm},\xi_{l{m}})\right]\\\label{lem1:e1c}
		&\qquad + \frac{1}{m^2}\sum_{l,l'=1}^m  \kappa(\xi_{l'm},\cdot)\kappa(\xi_{l'm},\xi_{l{ m}})\tilde{v}_{jm}(\xi_{lm}).
	\end{align}
	Now, for \eqref{lem1:e1c} we have
	\begin{align*}
		&\frac{1}{m^2}\sum_{l,l'=1}^m \kappa(\xi_{l'm},\cdot)\kappa(\xi_{l'm},\xi_{lm})\tilde{v}_{jm}(\xi_{lm})	= \frac{1}{m^2}\sum_{l,l'=1}^m  \kappa(\xi_{l'm},\cdot)\kappa(\xi_{l'm},\xi_{l{m}})\frac{1}{\tilde{\sigma}_{jm}\sqrt{m}}\sum_{i=1}^m (\tilde{w}_{jm})_i \kappa(\xi_{im},\xi_{lm})\\
		&~=~\frac{1}{m^\frac{3}{2}}\sum_{l,l'=1}^m \kappa(\xi_{l'm},\cdot)\kappa(\xi_{l'm},\xi_{lm}) (\tilde{z}_{jm})_l
		=\frac{\tilde{\sigma}_{jm}}{\sqrt{m}} \sum_{l'=1}^m  (\tilde{w}_{jm})_{l'}\kappa(\xi_{l'm},\cdot)
		=\tilde{\sigma}_{jm}^2 \tilde{v}_{jm}(\cdot).
	\end{align*}
	Next we  bound \eqref{lem1:e1b} from above and obtain
	\begin{align*}
		&\left\|\frac{1}{m}\sum_{l=1}^m\tilde{v}_{jm}(\xi_{lm})\left[\int \kappa(y,\cdot) \kappa(y,\xi_{lm}){\rm d}y - \frac{1}{m}\sum_{l'=1}^m \kappa(\xi_{l'm},\cdot) \kappa(\xi_{l'm},\xi_{l{ m}})\right]\right\|\\
		&~~~\le~\frac{C_K^2}{6m^2}\frac{1}{m}\sum_{l=1}^m|\tilde{v}_{jm}(\xi_{lm})| = \frac{C_K^2}{6m^3} \sum_{l=1}^m\left|\frac{1}{\tilde{\sigma}_{jm}\sqrt{m}}\sum_{i=1}^m(\tilde{w}_{jm})_i\kappa(\xi_{im},\xi_{lm})\right|\\
		&~~~=~\frac{C_K^2}{6m^\frac{5}{2}} \sum_{l=1}^m\left|(\tilde{z}_{jm})_i\right|\le \frac{C_K^2}{6m^2}.
	\end{align*}
	To bound \eqref{lem1:e1a} we use for $\alpha\in\{0,1,2\}$ that
	\begin{equation*}
		\|\partial^\alpha \tilde{v}_{jm}\|_\infty = \left\| \frac{1}{\tilde{\sigma}_{jm}\sqrt{m}}\sum_{l=1}^m \partial^\alpha_y \kappa(\xi_{lm},\cdot) (\tilde{w}_{jm})_l\right\|_\infty\le \frac{1}{\tilde{\sigma}_{jm}} \frac{C_K}{\sqrt{m}}\sum_{l=1}^m |(\tilde{w}_{jm})_l|\le \frac{C_K}{\tilde{\sigma}_{jm}}
	\end{equation*}
	and thus
	\begin{align*}
		&\left\|\int \kappa(y,\cdot)\left[\int \kappa(y,z) \tilde{v}_{jm}(z){\rm d}z - \frac{1}{m}\sum_{l=1}^m \kappa(y,\xi_{lm}) \tilde{v}_{jm}(\xi_{lm})\right]{\rm d}y\right\|\\
		&~~~\le~ C_K \frac{\sup_y\|\partial^2_z(\kappa(y,z)\tilde{v}_{jm}(z))\|_\infty}{24 m^2}
		\le \frac{C_K^3}{6 \tilde{\sigma}_{jm} m^2} \le \frac{\sqrt{2}C_K^3}{6\sigma_{j} m^2}.
	\end{align*} 
	Since $\sigma_j \le C_K$ we get 
	\begin{align*}
		\left\|K^*K \tilde{v}_{jm}- \tilde{\sigma}_{jm}^{  2} \tilde{v}_{jm}\right\|\le \frac{C_K^2}{6m^2}+\frac{\sqrt{2}C_K^3}{6\sigma_j m^2}\le \frac{(1+\sqrt{2})C_K^3}{6\sigma_j m^2} \le \frac{C_K^3}{2 \sigma_j m^2}
	\end{align*}
	which shows \eqref{lem1:e1}. Now we use Lemma \ref{lem1}. By Lemma \ref{lem00} we have that $|\sigma_j^2-\tilde{\sigma}_{jm}^2|\le \frac{C_K^2}{3m^2}< \frac{{c_j}}{2}$, so {  $j\in\arg\min_{i\in\N}\left|\sigma_i^2-\tilde{\sigma}_{jm}^2\right|=\{\psi_-(j),...,\psi_+(j)\}=:I_j$.} Then by \eqref{lem1:e02} of Lemma \ref{lem1} with
	\begin{equation*}
		\min_{i\not\in I_j}|\sigma_i^2-\tilde{\sigma}_{jm}^2|\ge \min_{i\not\in  I_j}|\sigma_i^2-\sigma_j^2| - |\sigma_j^2-\tilde{\sigma}_{jm}^2| \ge c_j - \frac{C_K^2}{3m^2} >\frac{c_j}{2}=:c
	\end{equation*}
	and 
	\begin{equation*}
		\|K^*K \tilde{v}_{jm}- \tilde{\sigma}_{jm}^2 \tilde{v}_{jm}\| \le \frac{C_K^3}{2\sigma_j m^2} =:\varepsilon,
	\end{equation*}
	{see \eqref{lem1:e1},} we deduce that { 
	\begin{align*}
		(\tilde{v}_{jm},P\tilde{v}_{jm}) &\ge \|\tilde{v}_{jm}\|^2 -\frac{\varepsilon^2}{c^2} = \|\tilde{v}_{jm}\|^2-\frac{\frac{C_K^6}{4\sigma_j^2 m^4}}{\frac{c_j^2}{4}}= \|\tilde{v}_{jm}\|^2 -\frac{C_K^6}{c_j^2\sigma_j^2 m^4}.
	\end{align*}}\hfill
\end{proof}
	\begin{proof}[Proof of Theorem \ref{t4}.]
 First, note that
  \begin{align*}
    \frac{(g_m,\tilde{w}_{jm})_{\R^m}}{\sqrt{m}\tilde{\sigma}_{jm}}&=\frac{\left(K_m f,\tilde{w}_{jm}\right)_{\R^m}}{\tilde{\sigma}_{jm}} = \frac{(f, K_m^* \tilde{w}_{jm})}{\tilde{\sigma}_{jm}} = \frac{\left(f,\sum_{l=1}^m (\tilde{w}_{jm})_l \kappa(\xi_{lm},\cdot)\right)}{\sqrt{m}\tilde{\sigma}_{jm}} = (f,\tilde{v}_{jm}).
   \end{align*}	
   We start to bound the different contributions of the decomposition \eqref{s4e0}.	First, we analyze the first term, the random contribution. { By \eqref{jm} and \eqref{lem002:e1} we have $\|\tilde{v}_{jm}\|^2\le 2$, and since the noise is unbiased and  $(\tilde{w}_{jm},\tilde{w}_{j'm})_{\R^m}=\delta_{jj'}$}, we obtain
	\begin{align*}
		&\E\left[\left\| \sum_{j=1}^k \frac{(g_m^\delta-g_m,\tilde{w}_{jm})_{\R^m}}{\sqrt{m}\tilde{\sigma}_{jm}} \tilde{v}_{jm}\right\|^2\right]= \sum_{j,j'=1}^k \frac{\E\left[(g^\delta_m-g_m,\tilde{w}_{jm})_{\R^m}(g^\delta_m-g_m,\tilde{w}_{j'm})_{\R^m}\right]}{m\tilde{\sigma}_{jm}\tilde{\sigma}_{j'm}} (\tilde{v}_{jm},\tilde{v}_{j'm})\\
		 &~~~=~\sum_{j,j'=1}^k \frac{\delta^2}{m\tilde{\sigma}_{jm}\tilde{\sigma}_{j'm}} (\tilde{w}_{jm},\tilde{w}_{j'm})_{\R^m}(\tilde{v}_{jm},\tilde{v}_{j'm}) =  \frac{\delta^2}{m}\sum_{j=1}^k \frac{1}{\tilde{\sigma}_{jm}^2}\|\tilde{v}_{jm}\|^2 \le \frac{4\delta^2}{m} \sum_{j=1}^k\frac{1}{\sigma_j^2}.
	\end{align*}	
{ 	Now, with Lemma \ref{lem002}, we have
	\begin{align}\label{t2:esttilde}
		&\|\tilde{v}_{jm}-P\tilde{v}_{jm}\|^2= \|\tilde{v}_{jm}\|^2 -2(\tilde{v}_{jm},P\tilde{v}_{jm})+\|P\tilde{v}_{jm}\|^2 = \|\tilde{v}_{jm}\|^2-(\tilde{v}_{jm},P\tilde{v}_{jm})\le \frac{C_K^{6}}{c_j^2\sigma_j^2 m^4}.
	\end{align}
 So
	\begin{align*}
		&\left\| \sum_{j=1}^k(f,\tilde{v}_{jm}-P\tilde{v}_{jm}) \tilde{v}_{jm}\right\|^2\\
		&~~~=~\sum_{j=1}^k (f,\tilde{v}_{jm}-P\tilde{v}_{jm})^2 \|\tilde{v}_{jm}\|^2 + \sum_{\substack{j,j'=1\\ j\neq j'}}^k (f,\tilde{v}_{jm}-P\tilde{v}_{jm})(f,\tilde{v}_{j'm}-P\tilde{v}_{j'm})(\tilde{v}_{jm},\tilde{v}_{j'm})\\
		&~~~\le~  \|f\|^{ 2}\sum_{j=1}^k\|\tilde{v}_{jm}-P\tilde{v}_{jm}\|^2\|\tilde{v}_{jm}\|^2 +\sum_{\substack{j,j'=1\\j\neq j'}}^k|(f,\tilde{v}_{jm}-P\tilde{v}_{jm})(f,\tilde{v}_{j'm}-P\tilde{v}_{j'm})| \frac{2C_K^2}{3 \sigma_j\sigma_{j'} m^2}\\
		&~~~\le~ 2\|f\|^{ 2} \sum_{j=1}^k \|\tilde{v}_{jm}-P\tilde{v}_{jm}\|^2 + \frac{2C_K^2}{3 m^2} \left(\sum_{j=1}^k \frac{|(f,\tilde{v}_{jm}-P\tilde{v}_{jm})|}{\sigma_j}\right)^2\\
		&~~~\le~ 2\|f\|^{2}\sum_{j=1}^k \|\tilde{v}_{jm}-P\tilde{v}_{jm}\|^2 + \frac{2 C_K^2\|f\|^2}{3m^2}\left(\sum_{j=1}^k\frac{\|\tilde{v}_{jm}-P\tilde{v}_{jm}\|}{\sigma_j}\right)^2\\
		&~~~\le~ \frac{ C_K^{ 6} \|f\|^2}{m^4}\sum_{j=1}^k\frac{1}{c_j^2\sigma_j^2 } + \frac{2C_K^{{8}} \|f\|^2}{3m^6}\left(\sum_{j=1}^k\frac{1}{c_j \sigma_j^2}\right)^2.
	\end{align*}
	Furthermore, using Cauchy-Schwarz, we get
	\begin{align*}
		\left\|\sum_{j=1}^k(f,P\tilde{v}_{jm})(\tilde{v}_{jm}-P\tilde{v}_{jm})\right\|^2 &\le \left(\sum_{j=1}^k|(f,P\tilde{v}_{jm})|\|\tilde{v}_{jm}-P\tilde{v}_{jm}\|\right)^2 \le \sum_{j=1}^k (f,P\tilde{v}_{jm})^2 \sum_{j=1}^k \|\tilde{v}_{jm}-P\tilde{v}_{jm}\|^2\\
		&\le \frac{C_K^{{6}}}{m^4}\sum_{j=1}^k(f,P\tilde{v}_{jm})^2\sum_{j=1}^k\frac{1}{c_j^2\sigma_j^2},
	\end{align*}
  and we can derive the bound
    \begin{align}\notag
        \sum_{j=1}^k(f,P\tilde{v}_{jm})^2 &= \sum_{j=1}^k\left(\sum_{j'=\psi_-(j)}^{\psi_+(j)}(f,v_{j'})(v_{j'},P\tilde{v}_{jm})\right)^2 \le \sum_{j=1}^k \sum_{j'=\psi_-(j)}^{\psi_+(j)}(f,v_{j'})^2 \sum_{j'=\psi_-(j)}^{\psi_+(j)}(v_{j'},P\tilde{v}_{jm})^2\\\label{t4:upbound}
        &\le \sum_{j=1}^k \sum_{j'=\psi_-(j)}^{\psi_+(j)}(f,v_{j'})^2 \|P\tilde{v}_{jm}\|^2 \le 2M_k \|f\|^2.
    \end{align}
    Finally,
    \begin{align*}
    \sum_{j=1}^k(f,P\tilde{v}_{jm})P\tilde{v}_{jm} - f&=\sum_{j=1}^{\psi_+(k)}(f,P\tilde{v}_{jm})P\tilde{v}_{jm} - \sum_{j=1}^{\psi_+(k)}(f,v_j)v_j - \sum_{j=\psi_+(k)+1}^\infty(f,v_j)v_j\\
    &\qquad- \sum_{j=k+1}^{\psi_+(k)} (f,P\tilde{v}_{jm})P\tilde{v}_{jm},
    \end{align*}
    and we set $k':=\psi_+(k)$. Since $k\le J_m$, also $\psi_+(k)\le J_m$ by definition of $J_m$. Note that $P\tilde{v}_{1m},...,P\tilde{v}_{k'm}$ are not orthonormal in general, so we define the Gramian matrix $\Gamma_{k'}:=\left((P\tilde{v}_{jm},P\tilde{v}_{j'm})\right)_{j,j'=1}^{k'}\in\R^{k'\times k'}$. From the orthonormality of the $v_i$ it follows that 
    \begin{align*}
&\left(P\tilde{v}_{jm},P\tilde{v}_{j'm}\right)=\left(P\tilde{v}_{jm},P\tilde{v}_{j'm}\right)\delta_{\psi_-(j)\psi_-(j')}\\
&~~=\left(\tilde{v}_{jm},\tilde{v}_{j'm}\right)\delta_{\psi_-(j)\psi_-(j')}+\left(P\tilde{v}_{jm}-\tilde{v}_{jm},\tilde{v}_{j'm}\right)\delta_{\psi_-(j)\psi_-(j')}+\left(P\tilde{v}_{jm},P\tilde{v}_{j'm}-\tilde{v}_{j'm}\right)\delta_{\psi_-(j)\psi_-(j')}.
    \end{align*}
    From \eqref{t2:esttilde} we deduce
    \begin{align*}
    |(P\tilde{v}_{jm},P\tilde{v}_{j'm}-\tilde{v}_{j'm})|\delta_{\psi_-(j)\psi_-(j')}+ |(P\tilde{v}_{jm}-\tilde{v}_{jm},\tilde{v}_{j'm})|\delta_{\psi_-(j)\psi_-(j')}\le \frac{4C_K^{{3}}}{c_{j}\sigma_{j} m^2}.
    \end{align*}
    By the Gerschgorin circle theorem and Lemma \ref{lem002}, for the eigenvalues $\gamma_{1k'}, ...,\gamma_{k'k'}$ of $\Gamma_{k'}$, it holds that 
    \begin{align}\notag
    &|\gamma_{ik'}-1|\le |1-(P\tilde{v}_{im},P\tilde{v}_{im})| + \sum_{\substack{j=\psi_-(i)\\ j \neq i}}^{\psi_+(i)}|(P\tilde{v}_{im},P\tilde{v}_{jm})|\\\label{gramian}
    &~~\le|1-(\tilde{v}_{im},\tilde{v}_{im})| + \sum_{\substack{j=\psi_-(i)\\ j \neq i}}^{\psi_+(i)}|(\tilde{v}_{im},\tilde{v}_{jm})| + \frac{4 M_{{ i}} C_K^{{3}}}{c_i\sigma_i m^2} \le  \frac{5M_{{ i}} C_K^{{3}}}{c_i\sigma_i m^2},
    \end{align}
where we have used that $\frac{C_K^2}{3\sigma_i^2 m^2}\le \frac{C_K^3}{3\sigma_i^3 m^2} \le \frac{C_K^3}{3\sigma_ic_i m^2}$ in the last step. The choice of $J_m$ guarantees that the above estimate can further be bounded by $1/2$ from above, which implies that  all $\gamma_{ik'}$ are positive. Hence $\Gamma_{k'}$ is invertible and  $P\tilde{v}_{1m},...,P\tilde{v}_{{k'}m}$ are linearly independent. Since $k'=\psi_+(k)$ it follows  that ${\rm span}(P\tilde{v}_{1m},...,P\tilde{v}_{k'm})={\rm span}(v_1,...,v_{k'})$. We now define 
\begin{equation*}
W_{k'}: \R^{k'} \to L^2(0,1),~\alpha=\begin{pmatrix} \alpha_1 & ... & \alpha_{k'} \end{pmatrix}^T\mapsto \sum_{j=1}^{k'} \alpha_j P\tilde{v}_{jm}
\end{equation*}
with adjoint
\begin{equation*}
{W_{k'}}^*:L^2(0,1)\to \R^{k'},~h\mapsto \begin{pmatrix} (h,P\tilde{v}_{1m}) & ... & (h,P\tilde{v}_{{k'}m})\end{pmatrix}^T.
\end{equation*}
Then $\Gamma_{k'}\alpha={W_{k'}}^*W_{k'}\alpha$ and it is easy to check that $W_{k'}\Gamma_{k'}^{-1}W_{k'}^*$ is a bounded linear and normal projection, therefore $W_{k'}\Gamma_{k'}^{-1}W_{k'}^*  = P_{{\rm span}(P\tilde{v}_{1m},...,P\tilde{v}_{k'm})}=P_{{\rm span}(v_1,...,v_{k'})}$. Thus, it holds that
\begin{equation*}
\sum_{j=1}^{k'}(f,P\tilde{v}_{jm})P\tilde{v}_{jm} - \sum_{j=1}^{k'}(f,v_j)v_j = W_{k'}W_{k'}^*f - W_{k'} \Gamma_{k'}^{-1} W_{k'}^*f= W_{k'}(I_{k'}-\Gamma_{k'}^{-1})W_{k'}^*f,
\end{equation*}
with $I_{k'}:\R^{k'}\to \R^{k'}$ being the identity matrix. Using \eqref{t4:upbound} with $h$ instead of $f$, we obtain the estimate
\begin{equation*}
    \|W_{k'}\|=\|W_{k'}^*\|=\sup_{\substack{ h\in L^2(0,1)\\ \|h\|=1}}\left\|\begin{pmatrix} (h,P\tilde{v}_{1m}) \\ ... \\ (h,P\tilde{v}_{{k'}m})\end{pmatrix}\right\|_{\R^m} = \sup_{\substack{ h\in L^2(0,1)\\ \|h\|=1}}\sqrt{\sum_{j=1}^{k'}(h,P\tilde{v}_{jm})^2} \le   \sqrt{2 M_{k'}}.
\end{equation*}
Furthermore, \eqref{gramian} implies that 
\begin{equation*}
    \left|1- \gamma_{i{k'}}^{-1}\right| = \frac{\displaystyle|1-\gamma_{i{k'}}|}{\displaystyle\gamma_{i{k'}}} \le \frac{\displaystyle\frac{\displaystyle5M_{ i} C_K^{{3}}}{\displaystyle c_i\sigma_i m^2}}{\displaystyle1-\frac{\displaystyle5M_{ i} C_K^{{3}}}{\displaystyle c_i\sigma_i m^2}}    \le \frac{\displaystyle10M_{{ i}}C_K^{{3}}}{\displaystyle c_i\sigma_i m^2},
\end{equation*}
and since $I_{k'}$ is the identity matrix, $\|I_{k'}-\Gamma_{k'}^{-1}\|\le \max_{i\le k'}\frac{10M_{i} C_K^{{3}}}{c_{i}\sigma_{i} m^2}$ follows. Consequently,
\begin{equation*}
\left\|\sum_{j=1}^{k'}(f,P\tilde{v}_{jm})P\tilde{v}_{jm} - \sum_{j=1}^{k'}(f,v_j)v_j\right\| =\| W_{k'}(I_{k'}-\Gamma_{k'}^{-1})W_{k'}^*f\| \le { M_{k}} \max_{i\le k'}\frac{20 M_{i} C_K^{{3}}\|f\|}{c_{i}\sigma_{i} m^2}.
\end{equation*}
It remains to bound $\sum_{j=k+1}^{k'}(f,P\tilde{v}_{jm})P\tilde{v}_{jm}.$ With similar reasoning as in \eqref{t4:upbound} we obtain
\begin{align*}   \left\|\sum_{j=k+1}^{k'}(f,P\tilde{v}_{jm})P\tilde{v}_{jm}\right\|^2&\le \sum_{j=k+1}^{k'}(f,P\tilde{v}_{jm})^2 \sum_{j=k+1}^{k'} \|P\tilde{v}_{jm}\|^2\le 4 M_k^2 \sum_{j=\psi_-(k)}^{k'}(f,v_j)^2. 
\end{align*}
    	Combining all the previous estimates, we end up with 
	\begin{align*}
		&\sqrt{\E\|\tilde{f}_{k,m}^\delta-f\|^2}\le \frac{2\delta}{\sqrt{m}}\sqrt{\sum_{j=1}^k\frac{1}{\sigma_j^2}}+ \frac{(1+\sqrt{{ 2}M_k})C_K^{{3}}\|f\|}{m^2}\sqrt{\sum_{j=1}^k\frac{1}{c_j^2 \sigma_j^2}} + \frac{\sqrt{2} C_K^{{4}} \|f\|}{\sqrt{3}m^3}\sum_{j=1}^k\frac{1}{c_j\sigma_j^2}\\
  &\qquad+ { M_k}{\max_{i\le k'}\frac{20 M_{i}C_K^{3}\|f\|}{c_{i}\sigma_{i} m^2}}+ 2 M_k \sqrt{\sum_{j=\psi_-(k)}^{\psi_+(k)}(f,v_j)^2}+\sqrt{\sum_{j=\psi_+(k)+1}^\infty (f,v_j)^2}.
	\end{align*}
} \hfill
	\end{proof} 
\begin{proof}[Proof of Theorem \ref{t6}.]
		As mentioned above, compared to the setting in Theorem \ref{t4}, we need to carefully analyze the contribution of the systematic component in the data propagation error. First, we show that
\begin{equation}\label{t6e1}
		\left\|\overline{g}_{m_o}-g_{m_o}\right\|^2\le \frac{\|g''\|_\infty^2}{9\cdot 64 m_o^3}.
	\end{equation}
To prove \eqref{t6e1}, we check that the right-hand side $g$ of \eqref{s1:e0} is twice differentiable. In fact, we have
	\begin{align*}
		\frac{d^2}{dx^2}g(x) =\frac{d^2}{d x^2} (Kf)(x) = \frac{d^2}{d x^2} \int \kappa(x,y) f(y) {\rm d}y = \int \frac{\partial^2}{\partial x^2} \kappa(x,y) f(y){\rm d}y.
	\end{align*}
	This is due to the dominated convergence theorem, since $\kappa$ is twice differentiable and $|f|$ is integrable. 
	Thus, using the Cauchy-Schwarz inequality, we obtain
	\begin{equation*}
		\|g''\|_\infty \le \sup_x \| \partial_x^2\kappa(x,\dot)\| \|f\|\le C_{{K}} \|f\|.
	\end{equation*}
	Now we use the Taylor expansion around $\xi_{im_o}$ with the exact Peano remainder term. This gives for $\zeta_t \in [\xi_{{(o(i-1) +t)}m},\xi_{im_o}]$ the identity
	\begin{align*}
		&\frac{1}{o}\sum_{t=1}^og(\xi_{{((i-1)o+t)}m}) - g(\xi_{im_o})\\
		&~~~=~-g(\xi_{im_o})+\frac{1}{o}\sum_{t=1}^o\left(g(\xi_{im_o}) + g'(\xi_{im_o})(\xi_{{((i-1)o+t)}m} - \xi_{im_o}) + \frac{g''(\zeta_t)}{2}\left(\xi_{{((i-1)o+t)}m} - \xi_{im_o}\right)^2\right)\\
		&~~~=~\frac{1}{o}\sum_{t=1}^o\left(g'(\xi_{im_o})(\xi_{{((i-1)o+t)}m} - \xi_{im_o}) + \frac{g''(\zeta_t)}{2}\left(\xi_{{((i-1)o+t)}m} - \xi_{im_o}\right)^2\right).
	\end{align*}
	Next, we show that
	\begin{equation}\label{t6e2}
		\frac{1}{o}\sum_{t=1}^og'(\xi_{im_o})(\xi_{{((i-1)o+t)}m} - \xi_{im_o}) = 0.
	\end{equation}
	First,
	\begin{align*}
		\xi_{i{m_o}} &= \frac{1}{o}\sum_{t=1}^o \xi_{{(o(i-1)+t)}m},\\
		\xi_{{((i-1)o+t)}m} - \xi_{im_o}&=\frac{2(o(i-1)+t)-1}{2m} - \frac{2i-1}{2\frac{m}{o}} =  \frac{2t-o-1}{2m},	
	\end{align*}
	for $i=1,...,m/o$ and $t=1,...,o$. Furthermore, we have
	\begin{equation*}		\xi_{{((i-1)o+t)}m} - \xi_{im_o} = -\left(\xi_{{((i-1)o+(o-t))}m}-\xi_{im_o}\right)	\end{equation*}
 for $i=1,...,m/o$ and $t=1,...,\lfloor o/2\rfloor$, and additionally, if $o$ is odd, 
		\begin{equation*}
		\xi_{{((i-1)o+\frac{o+1}{2})}m} - \xi_{im_o} = 0.\end{equation*}
  This shows \eqref{t6e2}. Consequently,	
  \begin{align*}&\frac{1}{o}\sum_{t=1}^og(\xi_{{((i-1)o+t)}m}) - g(\xi_{im_o}) = \frac{1}{o}\sum_{t=1}^o \frac{g''\left(\zeta_t\right)}{2} \frac{(2t-o-1)^2}{4m^2}\end{align*}
	and we get
	\begin{align*}
		\|\overline{g}_{m_o}-g_{m_o}\|^2 & = \sum_{i=1}^{m_o}\left(\frac{1}{o}\sum_{t=1}^og(\xi_{{((i-1)o+t)} m}) - g(\xi_{im_o})\right)^2 \le \frac{\|g''\|_\infty^2}{64 m^4o^2}\sum_{i=1}^{m_o}\left(\sum_{t=1}^o(2t-o-1)^2\right)^2\\
		&= \frac{\|g''\|_\infty^2}{64 m^4o^2} m_o\left(\frac{1}{3}o(o^2-1)\right)^2\le \frac{\|g''\|_\infty^2}{9\cdot 64 m_o^3},
	\end{align*}
which shows \eqref{t6e1}. Now, for the data propagation error, we derive
	\begin{align*}
&\E\left[\left\|\sum_{j=1}^k\frac{(\overline{g}^\delta_{m_o}-g_{m_o},\tilde{w}_{jm_o})_{\R^m}}{\sqrt{m_o}\tilde{\sigma}_{jm_o}} \tilde{v}_{jm_o}\right\|^2\right]\\
		&~~~=~\sum_{j=1}^k \frac{\E\left[(\overline{g}^\delta_{m_o}-g_{m_o},\tilde{w}_{jm_o})^2_{\R^m}\right]}{m_o\tilde{\sigma}_{jm_o}^2}\|\tilde{v}_{jm_o}\|^2\\
		 &\qquad\qquad+ \sum_{\substack{j,i \\ j\neq i}}^k \frac{\E\left[\left(\overline{g}^\delta_{m_o}-g_{m_o},\tilde{w}_{jm_o}\right)_{\R^m}\right]\E\left[(\overline{g}^\delta_{m_o}-g_{m_o},\tilde{w}_{im_o})_{\R^m}\right]}{m_o \tilde{\sigma}_{jm_o}{\tilde{\sigma}_{im_o}}} (\tilde{v}_{jm_o},\tilde{v}_{im_o})\\
		&~~~\le~ \frac{2}{m_o}\sum_{j=1}^k\frac{\frac{\delta^2}{o}+(\overline{g}_{m_o}-g_{m_o},\tilde{w}_{jm_o})_{\R^m}^{{2}}}{\tilde{\sigma}_{jm_o}^2} + \frac{C_K^2}{3 m_o^3} \sum_{\substack{i,j\\i\neq j}}^k \frac{(\overline{g}_{m_o}-g_{m_o},\tilde{w}_{jm_o})_{\R^m}(\overline{g}_{m_o}-g_{m_o},\tilde{w}_{im_o})_{\R^m}}{\tilde{\sigma}_{jm_o}{\sigma_j \tilde{\sigma}_{im_o} \sigma_i}}\\
		&~~~\le~\frac{4\delta^2}{m}\sum_{j=1}^k\frac{1}{\sigma_{j}^2} + \frac{4}{m_o}\sum_{j=1}^k\frac{(\overline{g}_{m_o}-g_{m_o},\tilde{w}_{jm_o})_{\R^m}^2}{\sigma_{j}^2}+ \frac{{ 2}C_K^2}{3m_o^3}\left(\sum_{j=1}^k\frac{(\overline{g}_{m_o}-g_{m_o},\tilde{w}_{jm_o})_{\R^m}}{\sigma_{j}^2}\right)^2\\
		&~~~\le~\frac{4\delta^2}{m}\sum_{j=1}^k\frac{1}{\sigma_j^2} + \frac{4}{m_o \sigma_k^2} \sum_{j=1}^k(\overline{g}_{m_o}-g_{m_o},\tilde{w}_{jm_o})^2_{\R^m}\\
		 &\qquad\qquad+ \frac{{2}C_K^2}{3 m_o^3} \left(\sum_{j=1}^k\frac{1}{\sigma_j^4} \right)\left(\sum_{j=1}^k(\overline{g}_{m_o}-g_{m_o},\tilde{w}_{jm_o})_{\R^m}^2\right)\\
		&~~~\le~\frac{4\delta^2}{m}\sum_{j=1}^k\frac{1}{\sigma_j^2} + \frac{4}{m_o \sigma_k^2} \|\overline{g}_{m_o}-g_{m_o}\|^2_{\R^m} + \frac{{ 2}C_K^2}{3 m_o^3} \left\|\overline{g}_{m_o}-g_{m_o}\right\|^2_{\R^m}\sum_{j=1}^k\frac{1}{\sigma_j^4}\\
		&~~~\le~\frac{4\delta^2}{m}\sum_{j=1}^k\frac{1}{\sigma_j^2} + \frac{\|g''\|_\infty^2}{3^2 \cdot 2^4m_o^4 \sigma_k^2}  + \frac{C_K^2\|g''\|^2_\infty}{{3^3 \cdot 2^5} m_o^6}\sum_{j=1}^k\frac{1}{\sigma_j^4}.
	\end{align*}
	For the remaining error terms, we simply replace $m$ with $m_o$ and obtain
	{ 
		\begin{align*}
		\sqrt{\E\left\| \overline{\tilde{f}}^\delta_{k,m_o}-f\right\|^2} &\le\frac{2\delta}{\sqrt{m}}\sqrt{\sum_{j=1}^k\frac{1}{\sigma_j^2}} + \frac{\|g''\|_\infty}{12m_o^2 \sigma_k}  + \frac{C_K\|g''\|_\infty}{12\sqrt{6}m_o^3}\sqrt{\sum_{j=1}^k\frac{1}{\sigma_j^4}}  + \sqrt{\sum_{j=\psi_+(k)+1}^\infty(f,v_j)^2}\\
    &\qquad+  2M_k \sqrt{\sum_{j=\psi_-(k)}^{\psi_+(k)} (f,v_j)^2}{+}  \frac{(1+\sqrt{{2}M_k})C_K^{3}\|f\|}{m_o^2}\sqrt{\sum_{j=1}^k\frac{1}{c_j^2 \sigma_j^2}} + \frac{\sqrt{2} C_K^{4} \|f\|}{\sqrt{3}m_o^3}\sum_{j=1}^k\frac{1}{c_j\sigma_j^2}\\
    &\qquad+ { M_k}\max_{i\le k'}\frac{20 M_{i} C_K^{3}\|f\|}{c_{i}\sigma_{i} m_o^2}.	
\end{align*}}
This completes the proof.\hfill
	\end{proof}
\section*{Acknowledgments}
Funded by the Deutsche Forschungsgemeinschaft under Germany's Excellence Strategy - GZ 2047/1, Projekt-ID 390685813 (Hausdorff Center for Mathematics, University of Bonn) and EXC-2046/1, Projekt-ID 390685689 (The Berlin Mathematics Research Center MATH+).

\bibliographystyle{abbrvnat}
\bibliography{references}


\end{document}